\documentclass{amsart}
\usepackage{amsmath,amssymb,amsthm}
\usepackage{mathtools}
\usepackage{graphicx}
\usepackage{tikz}
\mathtoolsset{showonlyrefs}
\usepackage[unicode]{hyperref}

\newtheorem{theorem}{Theorem}[section]
\newtheorem{lemma}[theorem]{Lemma}

\newtheorem{cor}[theorem]{Corollary}

\theoremstyle{definition}
\newtheorem{definition}[theorem]{Definition}

\numberwithin{equation}{section}

\usepackage{mleftright,xparse}
\NewDocumentCommand\xDeclarePairedDelimiter{mmm}
{%
	\NewDocumentCommand#1{som}{%
		\IfNoValueTF{##2}
		{\IfBooleanTF{##1}{#2##3#3}{\mleft#2##3\mright#3}}
		{\mathopen{##2#2}##3\mathclose{##2#3}}%
	}%
}
\xDeclarePairedDelimiter{\abs}{\lvert}{\rvert}
\xDeclarePairedDelimiter{\set}{\lbrace}{\rbrace}
\xDeclarePairedDelimiter{\norm}{\lVert}{\rVert}
\xDeclarePairedDelimiter{\scal}{\langle}{\rangle}

\newcommand{\eps}{\varepsilon}

\newcommand{\ph}{\varphi}
\newcommand{\la}{\lambda}
\newcommand{\Ups}{\Upsilon}

\newcommand{\R}{\mathbb{R}}
\newcommand{\N}{\mathbb{N}}

\newcommand{\f}[2]{\frac{#1}{#2}}
\newcommand{\tf}[2]{\tfrac{#1}{#2}}
\newcommand{\dd}[1]{\,\mathrm{d}{#1}}

\newcommand{\nn}{\nabla}

\newcommand{\on}{\quad\text{on }}
\newcommand{\inn}{\quad\text{in }}
\newcommand\restr[2]{{
		\left.\kern-\nulldelimiterspace 
		#1 
		\vphantom{\big|} 
		\right|_{#2} 
}}
\newcommand\rest[2]{{
		\left.\kern-\nulldelimiterspace 
		#1 
		\right|_{#2} 
}}

\DeclareMathOperator{\Div}{div}

\DeclareMathOperator{\tr}{tr}

\usepackage{stackengine}
\newcommand{\wcs}{\mathrel{\ensurestackMath{\stackon[0.45ex]{\rightharpoonup}{\scriptstyle\ast}}}}
\newcommand{\ccdot}{:}

\newcommand{\bs}{\boldsymbol}

\newcommand{\ve}{\bs v}

\newcommand{\Fe}{\bs F}
\newcommand{\we}{\bs w}

\newcommand{\he}{\bs h}

\newcommand{\uu}{\bs u}

\newcommand{\n}{\bs n}
\newcommand{\x}{\bs x}
\newcommand{\z}{\bs z}
\newcommand{\fe}{\bs f}
\newcommand{\af}{\bs \alpha}
\newcommand{\fit}{\bs \varphi}

\newcommand{\je}{\bs j}

\newcommand{\ta}{\bs \tau}

\newcommand{\X}{\bs X}
\newcommand{\E}{\bs{E\!X}}
\newcommand{\B}{\bs{B\!X}}
\newcommand{\PE}{E\!P}
\newcommand{\PB}{B\!P}

\newcommand{\Gd}{\Gamma_{\rm D}}
\newcommand{\Go}{\Gamma}
\newcommand{\Gn}{\Gamma_{\rm N}}
\newcommand{\ii}{\int_{\Omega}}

\newcommand{\io}{\int_{\Go}}

\newcommand{\Lp}[2]{L^{#1}(#2)}
\newcommand{\LP}[2]{\bs{L}^{#1}(#2)}
\newcommand{\Ep}[2]{E^{#1}(#2)}
\newcommand{\EP}[2]{\bs{E}^{#1}(#2)}

\title[Analysis of a nonlinear elliptic systems with jump on the interface]{Existence and qualitative theory for nonlinear elliptic systems  with a~nonlinear interface condition used in electrochemistry}
\author[M. Bathory]{Michal Bathory}
\address{Mathematical Institute of Charles University\\
	Faculty of Mathematics and Physics\\
	Charles University\\
	Sokolovsk\'a 83\\186 75 Praha 8\\Czech Republic}
\email{bathory@karlin.mff.cuni.cz}
\author[M. Bul\'i\v cek]{Miroslav Bul\'i\v cek}
\address{Mathematical Institute of Charles University\\
	Faculty of Mathematics and Physics\\
	Charles University\\
	Sokolovsk\'a 83\\186 75 Praha 8\\Czech Republic}
\email{mbul8060@karlin.mff.cuni.cz}
\author[O. Sou\v{c}ek]{Ond\v{r}ej Sou\v{c}ek}
\address{Mathematical Institute of Charles University\\
	Faculty of Mathematics and Physics\\
	Charles University\\
	Sokolovsk\'a 83\\186 75 Praha 8\\Czech Republic}
\email{soucek@karel.troja.mff.cuni.cz}

\thanks{This work was supported by the project No. 18-12719S financed by GA\v{C}R and by the project
SVV-2019-260455}

\begin{document}
\begin{abstract}
	We study a nonlinear elliptic system with prescribed inner interface conditions. These models are frequently used in physical system where the ion transfer plays the important role for example  in modelling of nano-layer growth or Li-on batteries. The key difficulty of the model consists of the rapid or very slow growth of nonlinearity in the constitutive equation inside the domain or on the interface. While on the interface, one can avoid the difficulty by proving a kind of maximum principle of a solution, inside  the domain such regularity for the flux is not available in principle since the constitutive law is discontinuous with respect to the spatial variable. The key result of the paper is the existence theory for these problems, where we require that  the leading functional satisfies either the delta-two or the nabla-two condition. This assumption is applicable in case of fast (exponential) growth as well as in the case of very slow (logarithmically superlinear) growth.
\end{abstract}

\maketitle

\section{Introduction}
This paper focuses on the existence and uniqueness analysis of nonlinear elliptic systems with general growth conditions that may have discontinuity on an inner interface which describes  the transfer of a certain quantity through this interface. To describe such problem mathematically, we consider a domain $\Omega\subset\R^d$, $d\geq2$, with Lipschitz boundary $\partial\Omega$ and with an inner interface $\Go$. The considered domain and the interface are shown in Figure~\ref{fig1} and we always have in mind a similar situation. We could also consider more interfaces inside of the domain $\Omega$ but it would not bring any additional mathematical difficulties so we restrict ourselves only to the situation depicted in Fig.~\ref{fig1}. Thus, that the domain $\Omega$ is decomposed into two parts $\Omega_1$ and $\Omega_2$ by the interface $\Gamma$ such that  $\Omega_i$ is also Lipschitz for $i=1,2$. Further, we assume that there is the Dirichlet part of the boundary  $\Gamma_D \subset \partial\Omega$ and the Neumann part $\Gamma_N \subset \partial \Omega$ and  we denote by $\n$ the unit normal vector on $\Gamma$, which is understood always  as the unit normal outward vector to $\Omega_1$ at $\Gamma$ (note that then $-\n$ is the unit outward normal vector to $\Omega_2$ on $\Gamma$). We also use the symbol $\n$ to denote the unit outward normal vector to $\Omega$ on $\partial \Omega$.

\tikzset{every picture/.style={line width=0.75pt}} 

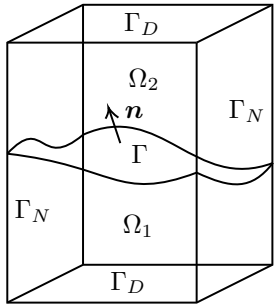
\begin{figure}\label{fig1}
	\centering
	\begin{tikzpicture}[x=0.75pt,y=0.75pt,yscale=-0.5,xscale=1]
	
	\draw    (317.2,50.58) -- (317.2,313.03) ;

	\draw    (412.7,50.58) -- (374.5,88.78) ;

	\draw    (279,88.78) -- (374.5,88.78) ;

	\draw    (374.5,88.78) -- (374.5,351.23) ;

	\draw   (317.2,50.58) -- (412.7,50.58) -- (412.7,50.58) -- (412.7,313.03) -- (374.5,351.23) -- (279,351.23) -- (279,351.23) -- (279,88.78) -- cycle ;
	\draw    (317.2,313.03) -- (279,351.23) ;

	\draw    (317.2,313.03) -- (412.7,313.03) ;

	\draw    (317.2,181.81) .. controls (357,146.23) and (373,240.23) .. (413,210.23) ;

	\draw    (279,200.91) .. controls (298,160.23) and (295,219.23) .. (317.2,181.81) ;

	\draw    (374.5,220.01) .. controls (392,236.23) and (402,238.23) .. (413,210.23) ;

	\draw    (279,200.91) .. controls (333,213.23) and (334.5,250.01) .. (374.5,220.01) ;

	\draw    (336,190.23) -- (330.33,156.21) ;
	\draw [shift={(330,154.23)}, rotate = 440.54] [color={rgb, 255:red, 0; green, 0; blue, 0 }  ][line width=0.75]    (10.93,-3.29) .. controls (6.95,-1.4) and (3.31,-0.3) .. (0,0) .. controls (3.31,0.3) and (6.95,1.4) .. (10.93,3.29)   ;
	
	\draw (348,126) node  [align=left] {$\displaystyle \Omega _{2}$};
	\draw (345,274) node  [align=left] {$\displaystyle \Omega _{1}$};
	\draw (345.85,200.91) node   {$\Gamma $};
	\draw (400,158) node   {$\Gamma _{N}$};
	\draw (292,259) node   {$\Gamma _{N}$};
	\draw (340,331) node   {$\Gamma _{D}$};
	\draw (347,70) node   {$\Gamma _{D}$};
	\draw (343,159.23) node   {$\boldsymbol{n}$};
	
	\end{tikzpicture}
	\caption{Prototypical domain $\Omega$.}
\end{figure}

The problem reads as follows: For given mappings $\he:\Omega\times \R^{d\times N} \to\R^{d\times N}$, $b:\Gamma\times \R^N\to\R^N$,   given Dirichlet data $\phi_0:\Gamma_D \to \R^N$ and Neumann data $\je_0:\Gamma_N \to \R^{d\times N}$, to find $\phi:\overline{\Omega} \to \R^N$ (here $N\in \mathbb{N}$ is a number of unknowns) solving the following system
\begin{equation}\label{pde}
\begin{aligned}
-\Div\he(x,\nn\phi(x))&=0 &&\inn\Omega,\\
\he(x,\nn\phi(x))\cdot\n(x)&=b(x,[\phi](x))&&\on\Go,\\
\he(x,\nn\phi (x))\cdot\n(x)&=\je_0 (x) \cdot \n(x) &&\on\Gn,\\
\phi(x)&=\phi_0(x) &&\on\Gd.
\end{aligned}
\end{equation}
Here, the symbol $[\phi]$ denotes the jump of $\phi$ on $\Go$. More precisely, for $x\in \Gamma$ we define
\begin{equation} \label{jumpfG}
[\phi](x):= \lim_{h \to 0_+} \phi(x+h\n(x))-\phi(x-h\n(x)).
\end{equation}
Consequently, we also cannot assume that $\phi$ has derivatives in the whole $\Omega$ and therefore the symbol $\nabla \phi$ appearing in \eqref{pde} is considered only in $\Omega_1$ and $\Omega_2$. Further, as we shall always assume that $\phi$ is a Sobolev function on $\Omega_1$ as well as on $\Omega_2$, it makes sense to talk about the trace of $\phi$ on $\partial \Omega_1$ and $\partial \Omega_2$  and thus the definition of $[\phi]$ is meaningful, see Section~\ref{Se2} for precise definitions and notations.

The model \eqref{pde} is frequently used when modelling the transfer of ions (or other particles) through the interface $\Gamma$ between two different materials with possibly different relevant properties represented by sets $\Omega_1$ and $\Omega_2$. The first prototypic example, we have in mind, is the the process of charging and discharging of lithium-ion batteries. The model of the form \eqref{pde} with $N=1$ and $\he$ being linear with respect to $\nabla \phi$ but being discontinuous with respect to $x$ when crossing the interface $\Gamma$ was derived and used for modelling of this phenomenon. Note that in this setting, the growth or behaviour of the function $b$ is very fast/wild, which may cause additional difficulties.  We refer to \cite{LaZa11,Lion,Seger} for physical justification of such a model and to \cite{dorfan,dorf} for the mathematical and numerical analysis of such model with zero $\je_0$. The second prototypic example is the modeling of porous metal oxide layer growth in the anodization process. The unknown function $\phi$ then represents an electrochemical potential. It has been experimentally observed that under some special conditions, the titanium oxide forms a nanostructure which resembles pores. In the thesis \cite{H}, it is confirmed numerically that the model \eqref{pde} (or rather its appropriate unsteady version) is able to capture this phenomenon if the nonlinearities $\he$ and $b$ are chosen accordingly. For this particular application,  the mapping  $\he$ models the high field conduction law in $\Omega_2$, while in $\Omega_1$ it corresponds to the standard Ohm law, and $b$ models the Butler-Volmer relation, see e.g.\ \cite{field2,H,field} and references therein for more details. After some unimportant simplifications and by setting all electrochemical constants equal to one, these electrochemical laws take the following form:
\begin{equation}\label{const}
\begin{aligned}
\he(x,\ve)&=\left\{
\begin{aligned}
&\f{\sinh{|\ve|}}{|\ve|}\ve &&\textrm{for }x\in \Omega_1, \, \ve\in\R^{d\times N},\\
&\ve &&\textrm{for }x\in \Omega_2, \, \ve\in\R^{d\times N},
\end{aligned}
\right.\\
b(x,z)&=\f{\exp|z|-1}{|z|}z \quad\textrm{for } z\in\R^N.
\end{aligned}
\end{equation}
Thus, it turns out that in some applications, the nonlinearities $\he$ and $b$ exhibit a very fast growth (exponential-like) with respect to the gradient of unknown and even worse due to the discontinuity with respect to the spatial variable the growth can oscillate between linear or exponential. Therefore, our aim is to obtain a reasonable mathematical  theory for \eqref{pde} under minimal assumptions on the smoothness with respect to the spatial variable $x$ and on the  growth with respect to the gradient of unknown required for $\he$ and $b$.

Without the interface condition on $\Go$, the system \eqref{pde} is a nonlinear elliptic system (provided that $\he$ is a monotone mapping), for which the existence theory can be obtained in a relatively standard way if $\he$ has polynomial growth and leads to the direct application of the standard monotone operator theory. Recently, this theory was further generalized in \cite{Bul2,Bul1} also into the framework of Orlicz spaces with $\he$ having a general (possibly exponential) growth and being discontinuous with respect to the spatial variable. The problem \eqref{pde} with the interface condition was also recently studied in \cite{dorfan} for the scalar setting, i.e. with $N=1$ and only for $\he$ being linear with respect to $\nabla \phi$ and having discontinuity with respect to the spatial variable on $\Go$. The authors in \cite{dorfan} established the existence of a weak solution for rather general class of functions $b$ describing the jump on the interface by proving the maximum principle for $\phi$. Note that such a procedure heavily relies on the scalar structure of the problem, the linearity of $\he$ is used in the proof and it also requires the zero flux $\je_0$.

To give the complete picture of the problem \eqref{pde}, we would like to point out that in case that  $\he$ and $b$ are strictly monotone (and consequently invertible), we can set $\fe:=\he^{-1}$ and $g:=b^{-1}$. Further, we denote $\je:=\he(\nn\phi)$, which in the electrochemical interpretation represents the current density flux. Then, the system \eqref{pde} can be rewritten as
\begin{equation}\label{pdem}
\begin{aligned}
-\Div\je (x)&=0 &&\inn\Omega,\\
\fe(x,\je(x))&=\nn\phi(x) &&\inn\Omega\setminus\Go,\\
g(x,\je (x)\cdot\n (x))&=[\phi](x) &&\on\Go,\\
\je(x)\cdot\n (x)&=\je_0(x)\cdot\n (x) &&\on\Gn,\\
\phi(x)&=\phi_0 (x) &&\on\Gd.
\end{aligned}
\end{equation}
and $\je:\Omega \to \R^{d\times N}$ can be seen as an  unknown. This is the first step to the so-called mixed formulation which seems to be  advantageous from the computational viewpoint, see the numerical experiments in \cite{H}.

The key result of the paper is that we provide a complete existence theory for model \eqref{pde} assuming very little assumption on the structure and growth of nonlinearities $\he$ and $b$ and on the data $\phi_0$ and $\je_0$ and we provide also its equivalence to \eqref{pdem}. Furthermore, we present a constructive proof based on the Galerkin approximation for both formulations \eqref{pde} and \eqref{pdem}, which may serve as a starting point for the numerical analysis. Moreover, in case that the nonlinearities are just derivatives of some convex potentials (which is e.g.\ the case of \eqref{const}), we show that the solution can be sought as a minimizer to certain functional. Finally, we would like to emphasize that we aim to build a robust mathematical theory for a very general class of problems allowing fast/slow growths of nonlinearities, minimal assumptions on data and being able to cover also general systems of elliptic PDE's, not only the scalar problem.

To end this introductory part, we just formulate a meta-theorem for the prototypic model \eqref{const} and refer to Section~\ref{Se2} for the precise statement of our result.
\begin{theorem}[Meta-theorem] \label{th:meta} Let the nonlinearities $\he$ and $b$ satisfy \eqref{const}. Then for any reasonable data $\phi_0$ and $\je_0$ there exist a unique solution $\phi$ to \eqref{pde} and a unique solution $\je$ to \eqref{pdem}. Moreover, these solutions can be found as minimizers of certain functionals.
\end{theorem}

\section{Notations \& Assumptions \& Results}\label{Se2}
In this part, we formulate precisely the main result of the paper. To do so rigorously, we first need to introduce certain function spaces that are capable to capture the very general behaviour of nonlinearities $\he$ and $b$. Therefore, we first shortly introduce the Musielak--Orlicz spaces, then we formulate the assumptions on nonlinearities $\he$ and $b$,  the geometry of $\Omega$ and the data $\phi_0$ and $\je_0$ and finally state the main results of the paper. Also we simply write the symbol ``$\cdot$" to denote the scalar product on $\R^d$ or just to say that the product has $d$-summands, whenever there is no possible confusion. Similarly, the symbol ``$\centerdot$" denotes the scalar product on $\R^N$ or the fact that the product has $N$-summands, and finally the symbol ``$\ccdot$" is reserved for the scalar product on $\R^{d\times N}$, or just for emphasizing that the product has $(d\times N)$-summands.

\subsection{The Musielak--Orlicz spaces}\label{M-O}
We recall here basic definitions and facts about Musielak--Orlicz spaces and the interested reader can find proofs e.g. in~\cite{Kras} or in a~book~\cite{HaHa}.

We say that $\Ups:\Omega \times \R^m\to[0,\infty)$ with  $m\in\N$, is an $N$-function if it is Carath\'{e}odory\footnote{The function $g(x,\z)$ is called Carath\'{e}odory if it is for almost all $x\in \Omega$ continuous with respect to $\z$ and also for all $\z \in \R^m$ measurable with respect to $x$.},  even and convex with respect to the second variable $\z\in \R^m$ and satisfies for almost all $x\in \Omega$ (note that this is a general definition but in our setting the number $m$ will correspond either to $N$ or to $d\times N$ depending on the context)
\begin{equation}\label{grow}
\lim_{|\z|\to 0}\f{\Ups(x,\z)}{|\z|}=0\quad\text{and}\quad\lim_{|\z|\to\infty}\f{\Ups(x,\z)}{|\z|}=\infty.
\end{equation}
Further, the $N$-function $\Ups$ is said to satisfy the $\Delta_2$ condition if there exist constants $c,K\in(0,\infty)$ such that for almost all $x\in \Omega$ and all $\z\in\R^m$ fulfilling $|\z|>K$ there holds
\begin{equation}\label{delta2}
\Ups(x,2\z)\leq c\,\Ups(x,\z).
\end{equation}
The complementary (convex conjugate) function to $\Ups$ is defined for all $(x,\z)\in \Omega\times \R^m$ by (within this section, the symbol ``$\cdot$" is also used for the scalar product on $\R^m$)
$$
\Ups^*(x,\z)=\sup_{\bs y\in\R^m}(\z\cdot{\bs y}-\Ups(x,\bs y))
$$
and it is also an $N$-function. This definition directly leads to the Young inequality
$$
\z_1\cdot \z_2\leq\Ups(x,\z_1)+\Ups^*(x,\z_2)\quad \textrm{for all } \z_1, \z_2\in\R^m
$$
and thanks to the convexity of $\Ups$ and the fact $\Ups(x,0)=0$ (it follows from \eqref{grow}), we have that for all $(x,\z)\in\Omega \times \R^m$ and $0<\eps<1$ there holds
$$
\Ups(x,\eps\z)\leq\eps\Ups(x,\z).
$$
This allows us to introduce the $\eps$-Young inequality (with $\eps \in (0,1)$)
$$
\z_1\cdot \z_2\leq\eps\Ups(x,\z_1)+\Ups^*\left(x,\f{\z_2}{\eps}\right).
$$

Having the notion of $N$-function, we can now define the Musielak--Orlicz spaces. Recall that $\Omega \subset\R^d$ is an open  set and for arbitrary $m\in \mathbb{N}$  define the set
$$
M^{\Ups}(\Omega):=\set{\ve\in L^1(\Omega; \mathbb{R}^m):\;\int_{\Omega}\Ups(x,\ve(x))\,\mathrm{d}x<\infty}.
$$
Since the set $M$ does not form necessarily a vector space, we define  the Orlicz space $L^{\Ups}(\Omega)$ as the linear hull of $M^{\Ups}(\Omega)$ and equip it with the  Luxembourg norm
$$
\norm{\ve}_{\Ups;\Omega}:=\inf\set{\la>0:\int_{\Omega}\Ups\left(x,\f{\ve(x)}{\lambda}\right)\dd{x}\leq 1}\quad \textrm{for all } \ve\in L^{\Ups}(\Omega).
$$
We will often omit writing the subscript $_\Omega$ whenever it is clear from the context.
It also directly follows from the Young inequality, that we have the H\"older inequality  in the form
$$
\int_{\Omega}\ve(x)\cdot\uu(x)\dd{x}\leq 2\norm{\ve}_{\Ups}\norm{\uu}_{\Ups^*}\quad\textrm{for all }\ve\in L^{\Ups}(\Omega)\textrm{ and all }\uu\in L^{\Ups^*}(\Omega).
$$
Note that the equality $M^{\Ups}(\Omega)=L^{\Ups}(\Omega)$ holds if and only if $\Ups$ satisfies the $\Delta_2$ condition \eqref{delta2}.
Further, by $E^{\Ups}(\Omega)$ we denote the closure of $L^{\infty}(\Omega; \mathbb{R}^m)$ in the norm $\norm{\cdot}_{\Ups}$. The purpose of this definition is that the space $E^{\Ups}(\Omega)$ is separable, since the set of all polynomials on $\Omega$ is dense in $E^{\Ups}(\Omega)$. In addition, if $\Ups$ satisfies the $\Delta_2$ condition, we have the following identities
\begin{equation}\label{2}
E^{\Ups}(\Omega)=M^{\Ups}(\Omega)=L^{\Ups}(\Omega),
\end{equation}
while if the $\Delta_2$ condition is not satisfied, there holds
\begin{equation}\label{2.5}
E^{\Ups}(\Omega)\subsetneqq M^{\Ups}(\Omega)\subsetneqq L^{\Ups}(\Omega).
\end{equation}
Furthermore, since  $E^{\Ups}(\Omega)$ is a linear space, we have for arbitrary $\ve \in E^{\Ups}(\Omega)$ and $K\in \R$ that $K\ve \in E^{\Ups}(\Omega)$. Consequently, it follows from \eqref{2}--\eqref{2.5} that
\begin{equation}
\ii \Ups(x,K\ve(x))\dd{x} < \infty \qquad \textrm{for all }\ve \in E^{\Ups}(\Omega) \textrm{ and all } K\in \R.\label{goodups}
\end{equation}
Finally, for any $N$-function $\Ups$, we have the following identification of dual spaces
\begin{equation}\label{5}
L^{\Ups}(\Omega)=(E^{\Ups^*}(\Omega))^*.
\end{equation}
Thus, although the space $L^{\Ups}(\Omega)$ is not reflexive\footnote{It is evident consequence of \eqref{2}, \eqref{2.5} and \eqref{5} that $L^{\Ups}(\Omega)$ is reflexive if and only if both functions $\Ups$ and $\Ups^*$ satisfy the $\Delta_2$ condition.} in general, the property \eqref{5} still ensures at least the weak$^\ast$ sequential compactness of bounded sets in $L^{\Ups}(\Omega)$ by the Banach-Alaoglu theorem. Finally, the space $\Lp{\Ups}{\Omega}$ coincides with the weak$^\ast$ closure of~$L^{\infty}(\Omega; \R^m)$.
%

The very similar definitions can be made for the spaces defined on $\Gamma$ (the $(d-1)$-dimensional subset of $\Omega$) and we have the spaces $E^{\Ups}(\Gamma)$, $M^{\Ups}(\Gamma)$ and $L^{\Ups}(\Gamma)$ with exactly same characterizations as above.

\subsection{Assumptions on the domain and nonlinearities}\label{S:ass}
We start this part by precise specification of the domain $\Omega$, whose prototype is depicted in Fig.~\ref{fig1}, where one can see $\Omega$ with its boundary $\partial\Omega=\Gd\cup\Gn$ and interface $\Go$. Below, we state precisely the necessary assumptions on $\Omega$, however the reader should always keep in mind the ``topology'' of the set from Fig.~\ref{fig1}.

\bigskip

\noindent{\bf Domain $\Omega$:} We assume the following:
\begin{itemize}
	\item[(O1)]The set $\Omega\subset\R^d$, $d\geq 2$, is open, bounded, connected and Lipschitz.
	\item[(O2)]The boundary $\partial\Omega$ can be written as a union of the closures of two relatively (in $(d-1)$ topology) open disjoint sets $\Gn$ and $\Gd$, where $\Gd$ consists of two separated components $\Gd^1$ and $\Gd^2$ of non-zero surface measure.
	\item[(O3)] The interface $\Go$ is a connected component of the set $\Omega$  that separates $\Gd^1$ from $\Gd^2$ such that the set $\Omega$ is bisected by $\Go$ into $\Omega_1$ and $\Omega_2$ and both $\Omega_1$ and $\Omega_2$ are Lipschitz sets.
\end{itemize}
We recall that the outward normal vector $\n$  on $\Go$ is chosen to point outwards $\Omega_1$.

Next, we introduce the assumptions on nonlinearities. We split them into two parts. The first one deals with the standard minimal assumption on the smoothness, growth and monotonicity, and the second one is an additional assumption that will be used for the existence theorem.

\bigskip

\noindent {\bf Assumptions on $\he$ and $b$:} We assume that $\he:\Omega\times \R^{d\times N} \to\R^{d\times N}$ and $b:\Gamma\times \R^N\to\R^N$ are Carath\'{e}odory mappings and satisfy:
\begin{itemize}
\item[(A1)] The mappings $\he$ and $b$ are  monotone with respect to the second variable and zero at zero, i.e. for all $\ve_1,\ve_2 \in \mathbb{R}^{d\times N}$, all $z_1,z_2 \in \mathbb{R}^N$ and almost all $x\in \Omega$ there holds
\begin{equation}\label{nula}
\begin{split}
(\he(x,\ve_1)-\he(x,\ve_2))\ccdot (\he_1 - \he_2) &\ge 0,\\
(b(x,z_1)-b(x,z_2))\centerdot (z_1 - z_2) &\ge 0,\\
\he(x,0)=b(x,0)&=0.
\end{split}
\end{equation}
\item[(A2)] There exist $N$-functions $\Phi$ and $\Psi$,  a nonnegative constant $C$ and positive constants $0<\alpha_{\he}, \alpha_b\leq1$ such that for all $\ve \in \mathbb{R}^{d\times N}$, all $z\in \mathbb{R}^N$ and almost all $x\in \Omega$, there holds
\begin{align}\label{i}
\he(x,\ve)\ccdot\ve&\geq \alpha_{\he}(\Phi^*(x,\he(x,\ve))+\Phi(x,\ve))-C,\\
\label{ii}
b(x,z)\centerdot z&\geq\alpha_b(\Psi^*(x,b(z))+\Psi(x,z))-C.
\end{align}
\end{itemize}

In case, we are more interested in the formulation for fluxes, i.e.\ for \eqref{pdem}, we have the following assumptions on $\fe$ and $g$.

\bigskip

\noindent {\bf Assumptions on $\fe$ and $g$:} We assume that $\fe:\Omega\times \R^{d\times N} \to\R^{d\times N}$ and $g:\Gamma\times \R^N\to\R^N$ are Carath\'{e}odory mappings and satisfy:
\begin{itemize}
\item[(A1)$^*$] The mappings $\he$ and $b$ are  monotone with respect to the second variable and zero at zero, i.e. for all $\ve_1,\ve_2 \in \mathbb{R}^{d\times N}$, all $z_1,z_2 \in \mathbb{R}^N$ and almost all $x\in \Omega$ there holds
\begin{equation}\label{nulaA}
\begin{split}
(\fe(x,\ve_1)-\fe(x,\ve_2))\ccdot (\ve_1 - \ve_2) &\ge 0,\\
(g(x,z_1)-g(x,z_2))\centerdot (z_1 - z_2) &\ge 0,\\
\fe(x,0)=g(x,0)&=0.
\end{split}
\end{equation}
\item[$\mathrm{(A2)}^*$] There exist $N$-functions $\Phi$ and $\Psi$,  a nonnegative constant $C$ and positive constants $0<\alpha_{\he}, \alpha_b\leq1$ such that for all $\ve \in \mathbb{R}^{d\times N}$, all $z\in \mathbb{R}^N$ and almost all $x\in \Omega$, there holds
\begin{align}
\label{iA}	\fe(x,\ve)\ccdot\ve&\geq \alpha_{\fe}(\Phi^*(x,\ve)+\Phi(x,\fe(x,\ve)))-C,\\
\label{iiA}	g(x,z)\centerdot z&\geq\alpha_g(\Psi^*(x,z)+\Psi(x,g(x,z)))-C.
\end{align}
\end{itemize}

Note that if $\he$ and $b$ are strictly monotone, i.e. \eqref{nula}$_1$ holds  for all $\ve_1\neq \ve_2$ with the strict inequality sign, then we can denote their inverses (with respect to the second variable) $\fe:=\he^{-1}$, $g:=b^{-1}$ and the assumptions (A1)--(A2) and (A1)$^*$--(A2)$^*$ are equivalent. Also the assumption $\he(0)=b(0)=0$ in (A1) is not necessary, it just makes the proofs more transparent. If $\he(0)\neq0$, we can always write $\he(\ve)=(\he(\ve)-\he(0))+\he(0)$ and follow step by step all proofs in the paper.

Finally, we specify the assumptions that will guarantee the existence (and also the uniqueness) of the solution to \eqref{pde} and \eqref{pdem}, respectively.

\bigskip

\noindent {\bf Key assumptions for the existence of solution:} In what follows we assume that at least one of the following holds:
\begin{itemize}
	\item[($\Pi$)] There exists $F_{\he}:\Omega \times \R^{d\times N} \to \R$ and $F_b:\Omega\times \R^N \to \R$ (potentials) such that $\he$ and $b$ are their Fr\'echet derivatives, i.e. for all $\ve \in \R^{d\times N}$, $z\in \R^N$ and almost all $x\in \Omega$ there hold
$$
\frac{\partial F_{\he}(x,\ve)}{\partial \ve}= \he(x,\ve), \qquad \frac{\partial F_b(x,z)}{\partial z}= b(x,z).
$$
\item[($\Pi$)$^*$] There exists $F_{\fe}:\Omega \times \R^{d\times N} \to \R$ and $F_{g}:\Omega\times \R^N \to \R$ (potentials) such that $\fe$ and $g$ are their Fr\'echet derivatives, i.e. for all $\ve \in \R^{d\times N}$, $z\in \R^N$ and almost all $x\in \Omega$ there holds
$$
\frac{\partial F_{\fe}(x,\ve)}{\partial \ve}= \fe(x,\ve), \qquad \frac{\partial F_{g}(x,z)}{\partial z}= g(x,z).
$$
\item[($\Delta$)]
At least one of the couples $(\Phi,\Psi)$ and $(\Phi^*,\Psi^*)$ satisfies\footnote{We say that a couple $(\Phi,\Psi)$ satisfies the $\Delta_2$ condition if both functions $\Phi$ and $\Psi$ satisfy the $\Delta_2$ condition.}  the $\Delta_2$ condition.
\end{itemize}

From now, whenever we talk about $\Phi$ and $\Psi$, we always mean the $N$-functions from \eqref{i}--\eqref{ii} or \eqref{iA}--\eqref{iiA}, respectively. Also to shorten the notation, we will omit writing the dependence on spatial variable $x\in \Omega$ but it is always assumed implicitly, e.g.  $\he(\ve)$ always means $\he(x,\ve)$ or $\he(x,\ve(x))$ depending on the context and similarly we use the same abbreviations for other functions/mappings.

\subsection{Notion of a weak solution} \label{S:weak}
In this part, we define the precise notion of a weak solution to \eqref{pde} and/or to \eqref{pdem}. Since we deal with functions that may have a jump across $\Gamma$, we use a slightly nonstandard definition of a weak gradient on $\Omega$, which however coincides with the standard definition on $\Omega_1$ and $\Omega_2$.  Therefore for any $q\in\Lp{1}{\Omega; \R^N}$, we  say that $\we \in \Lp{1}{\Omega; \R^{d\times N}}$ is a gradient of $q$ if for all $\fit\in \mathcal{C}^{\infty}_0 (\Omega\setminus\Go; \R^{d\times N})$ we have\footnote{For sake of clarity, the identity \eqref{grad} written in terms of components of $\we$, $\fit$ and $q$ has the following form
$$
\sum_{i=1}^N \sum_{j=1}^d \ii \we_{i,j} \fit_{i,j}=-\sum_{i=1}^N \ii q_i \left(\sum_{j=1}^d \frac{\partial \fit_{i,j}}{\partial x_j} \right).
$$}
\begin{equation}\label{grad}
\ii \we \ccdot \fit=-\ii q \centerdot (\Div\fit)
\end{equation}
and we will denote $\nn q:=\we$ as usual. This will be the default meaning of the symbol $\nn$ in the whole paper. It is easy to see that if $\nn q$ is integrable, then the restrictions $\rest{q}{\Omega_1}$ and $\rest{q}{\Omega_2}$ are  Sobolev functions on $\Omega_1$ and $\Omega_2$, respectively. Hence, since both sets are Lipschitz, we can define for such $q$'s the jump of $q$ across $\Go$ as
$$
[q]:=\restr{\tr_{\Omega_2}q}{\Go}-\restr{\tr_{\Omega_1}q}{\Go},
$$
where $\tr_{\Omega_i}$, $i=1,2$, is the trace operator acting upon functions defined on $\Omega_i$.

\bigskip

\noindent {\bf Function spaces related to problem \eqref{pde}.} First, we focus on the definition of certain spaces that are related to the problem \eqref{pde}. Thus, we introduce the following three spaces
\begin{align*}
P&:=\{q\in L^{1}(\Omega;\R^N):\nn q\in \LP{\Phi}{\Omega},\; [q]\in L^{\Psi}(\Go), \; \restr{\tr_{\Omega_1}q}{\Gd}=0,\;\restr{\tr_{\Omega_2}q}{\Gd}=0\},\\
\PE&:=\set{q\in P:\nn q\in\EP{\Phi}{\Omega},\quad[q]\in\Ep{\Psi}{\Go}},\\
\PB&:=\set{q\in P: \exists \{q^n\}_{n=1}^\infty \subset \PE, \; \nn q^n \rightharpoonup^* \nn q \textrm{ in }\LP{\Phi}{\Omega},\;  [q^n]\rightharpoonup^* [q] \textrm{ in }L^{\Psi}(\Go)}.
\end{align*}
We equip these spaces  with the norm
\begin{equation}\label{Banach}
\norm{q}_{P}:=\norm{\nn q}_{\Phi;\Omega}+\norm{[q]}_{\Psi,\Go},
\end{equation}
where the fact that it is a norm follows from the Poincar\'{e} inequality and from $|\Gamma_D|>0$. The motivation for definition of such spaces are the properties  of Musielak--Orlicz spaces stated in Section~\ref{M-O}. Moreover, we used the bold face to denote $\EP{\Phi}{\Omega}$ and $\LP{\Phi}{\Omega}$ to emphasize that the objects with values in $\R^{d\times N}$ are considered, while we used the normal font letters $L^{\Psi}(\Go)$ and $\Ep{\Psi}{\Go}$ to denote the space of mappings with value in $\R^N$. Furthermore,  the space $P$ equipped with the norm~\eqref{Banach} is a Banach space since it can be identified with  a closed subspace of the Banach space $\LP{\Phi}{\Omega}\times L^{\psi}(\Go)$ (see section Section~\ref{M-O} for properties of underlying spaces). However, since it is not separable in general, we construct the space $\PE$, which can be again identified with a closed subspace of $\EP{\Phi}{\Omega} \times \Ep{\Psi}{\Go}$, which is separable. Therefore the Banach space $\PE$ is separable as well. Finally, the fact, that the solution will be in most cases found as a weak$^*$ limit of functions from $\PE$,  motivates the definition of $\PB$, which is thus nothing else than the weak$^*$ closure of $\PE$. It is also evident that if $\Phi$ and $\Psi$ satisfy $\Delta_2$ condition then $P=\PE=\PB$.

\bigskip

\noindent{\bf Function spaces related to problem \eqref{pdem}.} In case we are more interested in solving~\eqref{pdem}, we set
\begin{align*}
\X&:=\{\ta\in \LP{\Phi^*}{\Omega}, \ta\cdot\n\in\Lp{\Psi^*}{\Go}: \int_{\Go} (\ta\cdot\n)\centerdot [\varphi]+\int_{\Omega}\nn\ph\ccdot\ta=0\; \forall \ph\in \PE\},\\
\E&:=\{\ta\in\X:\ta\in\EP{\Phi^*}{\Omega},\; \ta\cdot\n\in\Ep{\Psi^*}{\Go}\},\\
\B&:=\{\ta\in\X: \exists \{\ta^n\}_{n=1}^{\infty},\;  \ta^n \rightharpoonup^* \ta \textrm{ in } \EP{\Phi^*}{\Omega},\; \ta^n\cdot\n\rightharpoonup^* \ta\cdot\n \textrm{ in } \Ep{\Psi^*}{\Go}\}.
\end{align*}
Since we assume just integrability of $\ta :\Omega \to \R^{d\times N}$, we specify how the constraints from the definition of $\X$, $\E$ and $\B$ are understood. First, the meaning of divergence and the zero trace on the Neumann part of the boundary is usually  formulated as follows:
\begin{equation}\label{id1}
\left. \begin{aligned}
\ta\cdot\n&=0 &&\textrm{on }\Gamma_N\\
 \Div\ta&=0 &&\textrm{in }\Omega
\end{aligned}\right\}
\quad \overset{def}\Leftrightarrow \quad
\ii\nn\ph\ccdot\ta=0\quad\forall\ph\in \mathcal{C}^{0,1}(\overline{\Omega};\R^N), \; \rest{\ph}{\overline{\Gd}}=0.
\end{equation}
Note that the right hand side of \eqref{id1} is fulfilled for $\ta \in \X$ since Lipschitz functions vanishing on $\Gamma_D$ belong to $\PE$. Furthermore, these functions do not have a jump on~$\Gamma$ and therefore the corresponding integral in the definition of $\X$ vanishes.  Hence, \eqref{id1} is just the distributional form of  the operator $\Div$ (divergence) as well as  the trace of $\ta\cdot \n$. We just allow a broader class of test functions in the definition of $\X$. Second, we can specify the meaning of $\ta\cdot\n \in L^{\Psi^*}(\Go)$ in the definition of $\X$ as follows:
\begin{equation}
\ta\cdot\n \in L^{\Psi^*}(\Go)\overset{def}\Leftrightarrow  \exists w\in L^{\Psi^*}(\Go),\;
\int_{\Go}w\centerdot \varphi=\int_{\Omega_1}\nn\ph\ccdot\ta\;\forall\ph\in \mathcal{C}^{0,1}_0(\Omega;\R^N).\label{id2}
\end{equation}
Note that \eqref{id1} also implies that
$$
\int_{\Omega_1}\nn\ph\ccdot\ta=-\int_{\Omega_2}\nn\ph\ccdot\ta.
$$
Hence, since we know that $\rest{\ta \cdot \n}{\Go}$ is well defined distribution because $\Div \ta =0$, it follows from \eqref{id2} that $w$ can be identified with $\rest{\ta \cdot \n}{\Go}$, which is the meaning we use in the paper. However, also for the trace of $\ta \cdot \n$, we shall require a broader class of test functions than Lipschitz, which correspond to the test function from $\PE$ in the definition of $\X$. Finally, we equip $\X$, $\E$ and $\B$ with the norm
\begin{equation}\label{repr}
\norm{\ta}_{\X}:=\norm{\ta}_{\Phi^*;\Omega}+\norm{\ta\cdot\n}_{\Psi^*;\Go}.
\end{equation}
Similarly as before, we have that $\X$ and $\E$ are the Banach spaces and in addition, since $\E$ can be identified with a closed subspace of $\EP{\Phi^*}{\Omega} \times \Ep{\Psi^*}{\Go}$, which is separable, we have that $\E$ is separable as well.

\bigskip

\noindent{\bf Assumptions on data $\phi_0$ and $\je_0$.} The last set of assumptions is related to the given boundary and volume data. To simplify the presentation, we assume that $\phi_0$ and $\je_0$ are  defined on $\Omega$ and specify the assumptions\footnote{The reason for such simplification is that we do not want to employ the trace and/or the inverse trace theorem in Musielak--Orlicz spaces. But clearly, every $\phi_D\in W^{1,\infty}(\Gd)$ can be extended to the whole $\Omega$ such that it satisfies the assumption (D1).} on $\phi_0:\Omega \to \R^N$ and $\je_0:\Omega \to \R^{d\times N}$.
\begin{itemize}
\item[(D1)] We assume that $\phi_0\in W^{1,1}(\Omega; \R^N)$ such that
\begin{equation}\label{bc}
\nn\phi_0\in\EP{\Phi}{\Omega}.
\end{equation}
\item[(D2)] We assume that $\je_0:\Omega\to \R^{d\times N}$ is measurable and satisfies
\begin{equation}\label{bc2}
\je_0 \in \EP{\Phi^*}{\Omega}, \; \je_0\cdot\n=0 \textrm{ on }\Gamma,\; \Div\je_0=0 \textrm{ in } \Omega.
\end{equation}
\end{itemize}
It is worth noticing, that we assume here better properties than we expect from solution. First, since $\phi$ is a Sobolev function, it does not have any jump on $\Gamma$. Second, we assume the the flux $\je_0$ over the surface $\Gamma$ is also vanishing (since divergence is zero, we can talk about the normal component of the flux on $\Gamma$, see \eqref{bc2}). The reason for such setting is that we just want to simplify the presentation of main results  and the proofs.

\bigskip

\noindent {\bf Definition of a weak solution.}
We shall define four notions of weak solution - two for each formulation \eqref{pde} and \eqref{pdem}. We start with the motivation of a notion of weak solution to \eqref{pde}. We assume that we have a sufficiently good solution to \eqref{pde} and we take the scalar product of  the first equality (it has $N$ component) in \eqref{pde} by arbitrary $q\in \PE$. We integrate the result over $\Omega$ and after using integration by parts, we deduce that (recall our notation for $\nabla q$ in \eqref{grad} and also our definition of $\n$ and $[q]$ on $\Gamma$)
$$
\begin{aligned}
	0&= -\int_{\Omega_1}\Div(\he(\nn\phi)-\je_0) \centerdot q -\int_{\Omega_2}\Div(\he(\nn\phi)-\je_0) \centerdot q\\
	&=-\int_{\partial \Omega_1\setminus \Gamma}(\he(\nn\phi)-\je_0)\n \centerdot q -\int_{\partial \Omega_2\setminus \Gamma}(\he(\nn\phi)-\je_0)\n \centerdot q+\int_{\Gamma}(\he(\nn\phi)-\je_0)\n \centerdot [q]\\
&\qquad +\int_{\Omega}(\he(\nn\phi)-\je_0) \ccdot \nabla q \\
	&\underset{\eqref{bc2}}{\overset{\eqref{pde}}=}\int_{\Go}b([\phi]) \centerdot [q]+\int_{\Omega}\he(\nn\phi) \ccdot \nabla q-\int_{\Omega}\je_0\ccdot\nn q,
	\end{aligned}
$$
where we also used the facts that $q$ vanishes on $\Gamma_D$, that $\Div \je_0 =0$ and that $\je_0\cdot \n = 0$ on $\Gamma$.  The above identity can thus be understood as a weak formulation of \eqref{pde} and we are led to the following definition.

\begin{definition}\label{ws2} Let $\Omega$ satisfy (O1)--(O3), nonlinearities $\he$ and $b$ satisfy (A1)--(A2), data $\phi_0$ and $\je_0$ satisfy (D1)--(D2).  We say that the function $\phi$ is a weak solution to \eqref{pde} if
$$\phi-\phi_0\in P,\quad\he(\nn\phi)\in\LP{\Phi^*}{\Omega},\quad b([\phi])\in\Lp{\Psi^*}{\Go}$$
and
\begin{equation}\label{wspot}
\ii\he(\nn\phi)\ccdot\nn q+\int_{\Go}b([\phi])\centerdot[q]=\int_{\Omega} \je_0 \ccdot \nn q \quad \textrm{ for all } q\in \PE.
\end{equation}
\end{definition}
Using the H\"older inequality, we see that both integrals in \eqref{wspot} are well defined. In addition, we see that for sufficiently regular $\phi$, the computation above shows that the $\phi$ solving \eqref{wspot} solves \eqref{pde} as well. Further, we introduce another concept of solution, which a~priori does not require any information on $\he (\nabla \phi)$ and $b([\phi])$.
\begin{definition}\label{Den}
Let $\Omega$ satisfy (O1)--(O3), nonlinearities $\he$ and $b$ satisfy (A1)--(A2), data $\phi_0$ and $\je_0$ satisfy (D1)--(D2).  We say that the function $\phi$ is a variational weak solution to \eqref{pde} if
$$\phi-\phi_0\in P$$
and
\begin{equation}\label{ener}
\ii(\he(\nn\phi)-\je_0)\ccdot\nn(\phi-\phi_0 - q)+\int_{\Go}b([\phi])\centerdot[\phi-q]\le 0 \quad \textrm{ for all } q\in \PE.
\end{equation}
\end{definition}
Although,  we did not impose any assumptions on the integrability of $\he(\nn\phi)$ and $b([\phi])$, this information is included implicitly in \eqref{ener} as it is shown in Lemma~\ref{cons} below.

The next notion of a weak solution concerns the ``dual'' formulation \eqref{pdem} in terms of the flux $\je$. Formally, it can be again derived from \eqref{pdem}, \eqref{bc} and integration by parts as follows
\begin{align*}
\ii(\fe(\je)-\nn\phi_0)\ccdot\ta&=\int_{\Omega_1}\nn(\phi-\phi_0)\ccdot\ta+\int_{\Omega_2}\nn(\phi-\phi_0)\ccdot\ta\\
&=\int_{\partial\Omega_1}(\phi-\phi_0)\centerdot(\ta\cdot\n)+\int_{\partial\Omega_2}(\phi-\phi_0)\centerdot(\ta\cdot\n)\\
&=-\int_{\Go}[\phi]\centerdot(\ta\cdot\n)=-\int_{\Go}g(\je\cdot\n)\centerdot(\ta\cdot\n)
\end{align*}
for any $\ta\in\E$.

Thus, we are led to the following definition.
\begin{definition}\label{ws}
Let $\Omega$ satisfy (O1)--(O3), nonlinearities $\fe$ and $g$ satisfy (A1)$^*$--(A2)$^*$, data $\phi_0$ and $\je_0$ satisfy (D1)--(D2).  We say that the function $\je$ is a weak solution to \eqref{pdem} if
$$\je - \je_0 \in\X,\quad\fe(\je)\in\LP{\Phi}{\Omega},\quad g(\je\cdot\n)\in\Lp{\Psi}{\Go}$$
and
\begin{equation}\label{ws1}
\ii\fe(\je)\ccdot\ta+\int_{\Gamma}g(\je\cdot\n)\centerdot(\ta\cdot\n)=\ii\nn\phi_0\ccdot\ta\quad \textrm{ for all } \ta\in\E.
\end{equation}
\end{definition}

Analogously as for $\phi$, we can define the variational weak solution also for $\je$.
\begin{definition}\label{wsf}
Let $\Omega$ satisfy (O1)--(O3), nonlinearities $\fe$ and $g$ satisfy (A1)$^*$--(A2)$^*$, data $\phi_0$ and $\je_0$ satisfy (D1)--(D2).  We say that the function $\je$ is a variational weak solution to \eqref{pdem} if
	$$\je-\je_0\in\X$$
	and
	\begin{equation}\label{wsf1}
	\ii(\fe(\je)-\nn\phi_0)\ccdot(\je-\je_0-\ta )+\int_{\Gamma}g(\je\cdot\n)\centerdot((\je-\ta )\cdot\n) \le 0 \quad \textrm{ for all } \ta\in\E.
	\end{equation}
\end{definition}

Note that in Definition~\ref{ws2} the boundary condition $\phi=\phi_D$ on $\Gd$ is imposed by $\phi-\phi_0\in P$, whereas in Definition~\ref{ws} the same boundary condition is encoded in \eqref{ws1} implicitly (it is shown later, see part ii) of Theorem~\ref{eq}). The situation is reversed for the boundary condition $\je\cdot\n=\je_0\cdot\n$ on $\Gn$.

\section{Main results}

We start this section with the first  key result of the paper that focuses on the existence and uniqueness of a solution to \eqref{pde}.
\begin{theorem}\label{ex1}
	Let $\Omega$ satisfy {\rm(O1)--(O3)} and $\phi_D$ fulfil {\rm(D1)}. Suppose that $\he$ and $b$ satisfy {\rm(A1), (A2)}.
	\begin{itemize}
		\item[(i)] Assume that {\rm(}$\Delta${\rm)} holds. Then, there exists a weak solution $\phi$ to \eqref{pde}. In addition the weak solution satisfies $\phi \in \phi_0+\PB$ and  \eqref{wspot} and \eqref{ener} are valid for any function $q\in \PB$.
		\item[(ii)] Assume that $(\Pi)$ holds.
			Then, there exists a variational weak solution $\phi\in\phi_0+P$ to \eqref{pde} and this solution is also a weak solution.
\end{itemize}
If, in addition, the mapping $\he$ is strictly monotone, then the weak solution is unique in the class $\phi_0+\PB$.
\end{theorem}


As a direct consequence of the above theorem, we also obtain the result stated in Meta-theorem~\ref{th:meta}, which is now formulated as
\begin{cor}\label{coro}
	Let $\Omega$ satisfy {\rm (O1)--(O3)}, let $N=1$  and let $\phi_0$ and $\je_0$ fulfil {\rm(D1)} and {\rm (D2)} with $\Phi(\ve):=\cosh (|\ve|)-1$ and set $\Psi(z):= \exp(|z|)-|z|-1$. Then $\Phi$ and $\Psi$ are $N$-functions and  there exists a unique variational weak solution $\phi\in \phi_0 + \PB$ to
\begin{equation}	
\begin{aligned}
	\Div\left(\f{\sinh|\nn\phi|}{|\nn\phi|}\nn\phi\right)&=0 &&\inn\Omega\setminus\Go,\\
	\f{\sinh|\nn\phi|}{|\nn\phi|}\nn\phi\cdot\n-\f{\exp(|[\phi]|)-1}{|[\phi]|}[\phi]&=0  &&\on\Go,\\
	\nn\phi\cdot\n&=\je_0\cdot \n &&\on\Go_N,\\
	\phi&=\phi_0 &&\on\Gd.
	\end{aligned}
\end{equation}
\end{cor}

To summarize, we can obtain the existence of a weak solution in two cases. Either in case that there exists a potential (in this case the solution will be sought as a minimizer) or in case that  ($\Delta$) holds. Note that ($\Delta$)  is quite a weak assumption as the $N$-functions $\Phi$ such that both $\Phi$ and $\Phi^*$ do not satisfy the $\Delta_2$ condition are not that easy to find, especially in the applications (see the example in \cite[p.\ 28]{Kras}). Moreover, we would like to point out here that in case ($\Delta$) holds, we obtained a better solution than just $\phi\in\phi_0+P$ and we even have $\phi\in\phi_0+\PB$. Note that it is trivial if $\Phi$ and $\Psi$ satisfy the $\Delta_2$ condition. However, if it is not the case, it is a piece of new information. Second, we obtained the uniqueness in the class $\phi_0+\PB$, which may be a smaller class than that introduced for weak solution. However, since we know that there exists a weak solution in $\phi_0+\PB$, this class may be understood as a proper selector for obtaining a uniqueness of a solution.

The second existence theorem uses the alternative weak formulation \eqref{pdem} in terms of the flux $\je$.

\begin{theorem}\label{ex2}
	Let $\Omega$ satisfy {\rm(O1)--(O3)} and let $\je_0$ fulfil {\rm(D2)}. Suppose that $\fe$ and $g$ satisfy $\mathrm{(A1)}^*$ and $\mathrm{(A2)}^*$.
	\begin{itemize}
		\item[(i)] Assume that {\rm(}$\Delta${\rm)} holds.
			Then, there exists a weak solution $\je$  \eqref{pdem}. In addition the weak solution fulfills $\je \in \je_0 + \B$ and \eqref{ws1} and \eqref{wsf1} are valid for any function $\ta\in \B$.
		\item[(ii)] Assume that $(\Pi^*)$ holds.
		Then, there exists a variational weak solution $\je\in \je_0+\X$ to \eqref{pdem} and this solution is also a weak solution.
	\end{itemize}
If, in addition, the mapping $\fe$ is strictly monotone, then the weak solution is unique in the class $\je_0+\B$.
\end{theorem}
Also here, we would like to point out that in case ($\Delta$) holds, we found a solution in $\B$ and this is also the class of solutions in which we obtained the uniqueness.

Finally, we state the result about the equivalence of  Definitions~\ref{ws2}~and~\ref{ws}.
\begin{theorem}\label{eq}
Let all assumptions of Definitions~\ref{ws2}~and~\ref{ws} be satisfied. In addition, assume that $\he$, $\fe$, $g$ and $b$ are strictly monotone, satisfying $\he^{-1}=\fe$ and $b^{-1}=g$. Then
\begin{itemize}
\item[i)]If $\phi$ is a weak solution in sense of Definition~\ref{ws2} then $\je:=\he(\nn \phi)$ satisfies $\je-\je_0\in \X$ with  $\je\cdot \n= b([\phi])$ on $\Gamma$ and \eqref{ws1} holds for all $\ta\in \E \cap \mathcal{C}^1(\overline{\Omega};\mathbb{R}^{d\times N})$. In addition if {\rm(}$\Delta${\rm)} holds and $\phi \in \phi_0 + \PB$ then $\je$ is a weak solution in sense of Definition~\ref{ws}.
\item[ii)]If $\je$ is a weak solution in sense of Definition~\ref{ws} then there exists  $\phi\in \phi_0+P$ fulfilling $\nn \phi=\fe(\je)$ in $\Omega$ and $[\phi]=g(\je\cdot \n)$ on $\Gamma$ and $\phi$ is a weak solution in sense of Definition~\ref{ws2}.
\end{itemize}
\end{theorem}
This theorem shows the equivalence between the notions of solution if ($\Delta$) holds. Furthermore, if $(\Delta)$ is not satisfied then we have at least the equivalence of solution in class of distributional solutions of \eqref{pde} and \eqref{pdem} respectively. Furthermore, it follows from the above theorem, that we can choose the formulation, which is more proper e.g.\ for  numerical purposes, and we still construct the unique solution to the original problem. Moreover, we see that the existence of a weak solution $\je$ automatically implies the existence of a weak solution $\phi$ even in the case when $(\Delta)$ is not satisfied. Therefore also from the point of view of analysis of the problem, the dual formulation \eqref{pdem} seems to be preferable to the weak formulation \eqref{pde}.

The last result states when a variational weak solution is also a weak solution and similarly when a weak solution is also a variational weak solution.
\begin{theorem}\label{cons}
	Let $\phi\in \phi_0+P$ be a variational weak solution to \eqref{pde}. Then $\phi$ is also a weak solution and satisfies
	\begin{equation}\label{konecnost}
	\ii\Big(\Phi(\nn\phi)+\Phi^*(\he(\nn\phi))\Big)+\int_{\Go}\Big(\Psi([\phi])+\Psi^*(b([\phi]))\Big)<\infty.
	\end{equation}
Similarly, let $\phi\in \phi_0 + \PB$ be a weak solution to \eqref{pde} and $(\Delta)$ hold. Then $\phi$ is also a variational weak solution.

Let $\je\in \je_0+\X$ be a variational weak solution to \eqref{pdem}. Then $\je$ is also a weak solution and satisfies
	\begin{equation}\label{konec}
	\ii\Big(\Phi(\fe(\je))+\Phi^*(\je)\Big)+\int_{\Go}\Big(\Psi(g(\je\cdot\n))+\Psi^*(\je\cdot\n)\Big)<\infty.
	\end{equation}
Similarly, let $\je\in \je_0 + \B$ be a weak solution to \eqref{pdem} and $(\Delta)$ hold. Then $\je$ is also a variational weak solution.
\end{theorem}

In the rest of the paper, we prove the results stated in this section and finally give also the proof of Meta-theorem~\ref{th:meta}.

\section{Proofs of the main results}
This key part is organized as follows. First, in Section~\ref{SS1}, we show Theorem~\ref{cons}. Then in Section~\ref{SS2} we prove  Theorem~\ref{eq}. Sections~\ref{SS3}~and~\ref{SS4} are devoted to the proofs of Theorem~\ref{ex1} and \ref{ex2}, respectively. Since both proofs are almost identical, we prove Theorem~\ref{ex1} rigorously only for the case ii), i.e. if $(\Pi)$ holds, and Theorem~\ref{ex2} rigorously only for the case i), i.e. when $(\Delta)$ holds true. The corresponding counterparts of the proofs can be done in the very same way and therefore we present here only sketch of these proofs in Sections~\ref{SS5}~and~\ref{SS6}. Finally the proof of Corollary~\ref{coro} and consequently also of Meta-theorem~\ref{th:meta} is presented in Section~\ref{SS7}.

\subsection{Proof of Theorem~\ref{cons}}\label{SS1}
We start the proof by showing that variational weak solution is also weak solution.	Let $\phi\in \phi_0 + P$ be a variational weak solution. Thanks to the Young inequality and the assumption \eqref{i} (coercivity of $\he$), we can write
	\begin{align*}
	(&\he(\nn\phi)-\je_0)\cdot\nn(\phi-\phi_0)\\
	&\ge \alpha_{\he}\Phi^*(\he(\nn\phi))+\alpha_{\he}\Phi(\nn\phi)-D  -\frac{\alpha_{\he}\Phi^*(\he(\nn\phi))}{2} -\frac{\alpha_{\he}\Phi(\nn\phi)}{2}\\
&\quad - \Phi(\tf2{\alpha_{\he}}\nn \phi_0)-\Phi^*(\tf2{\alpha_{\he}}\je_0)- \Phi(\nn \phi_0)-\Phi^*(\je_0)\\
	&\ge \frac{\alpha_{\he}\Phi^*(\he(\nn\phi))}{2} +\frac{\alpha_{\he}\Phi(\nn\phi)}{2} - 2\Phi(\tf2{\alpha_{\he}}\nn \phi_0)-2\Phi^*(\tf2{\alpha_{\he}}\je_0)-D.
\end{align*}
Similarly, we also recall \eqref{ii}
\begin{equation*}
	\alpha_b\Psi([\phi])+\alpha_b\Psi^*(b([\phi]))\leq D+b([\phi])[\phi].
\end{equation*}
Then, we set $q:=0$ in \eqref{ener} and with the help of  above estimates we deduce that
\begin{align}\label{necoo}
\begin{aligned}
&\ii \Phi^*(\he(\nn\phi)) +\Phi(\nn\phi) + \int_{\Go}\Psi([\phi])+\Psi^*(b([\phi]))\\
&\qquad \le C\left(1+\ii \Phi(\tf2{\alpha_{\he}}\nn \phi_0)+\Phi^*(\tf2{\alpha_{\he}}\je_0) \right).
\end{aligned}
\end{align}
Since $\je_0 \in \EP{\Phi^*}{\Omega}$ and $\nabla \phi_0 \in \EP{\Phi}{\Omega}$, we can use  \eqref{goodups} and obtain  that the right hand side of \eqref{necoo} is finite. Hence, we obtain \eqref{konecnost}.

Thus, we just need to show that $\phi$ also satisfies  \eqref{wspot}. Note that thanks to \eqref{konecnost} all integrals in \eqref{wspot} and \eqref{ener} are well defined and finite. Let us define for arbitrary $q\in P$
$$
J(q):=\ii(\he(\nn\phi)-\je_0)\ccdot\nn q+\int_{\Go}b([\phi])\centerdot[q].
$$
Then, because we already have \eqref{konecnost}, we can rewrite \eqref{ener} as
$$
-\infty<J(\phi-\phi_0)\leq J(q)<\infty\qquad \textrm{for all }q\in \PE,
$$
which means that $J$ is bounded from below. But since $J$ is linear and $\PE$ is a linear space, this is  possible if and only if $J(q)=0$ for all $q\in \PE$, which is nothing else than \eqref{wspot}.

Next, we show that if $(\Delta)$ holds and a weak solution satisfies in addition $\phi\in \phi_0 +\PB$ then it is also a variational weak solution. Let us consider first the case when $\Psi$ and $\Phi$ satisfy $\Delta_2$ condition. Then $\PE=P$ and we can simply set $q:=\phi-\phi_0 - \tilde{q}$ in \eqref{wspot} with arbitrary $\tilde{q}\in \PE$ to obtain \eqref{ener} (where we replace $q$ by $\tilde{q}$). In the second case, i.e. if $\Psi^*$ and $\Phi^*$ satisfy $\Delta_2$ condition, we use the fact that $\phi-\phi_0\in \PB$. Thus, we can find a sequence $\{\phi^n-\phi_0\}_{n=1}^{\infty}\subset \PE$ such that
\begin{align}
\nn \phi^n-\nn\phi_0 &\rightharpoonup^* \nn\phi-\nn \phi_0 &&\textrm{ weakly$^*$ in } L^{\Phi}(\Omega),\label{prpr1}\\
[\phi^n] &\rightharpoonup^* [\phi] &&\textrm{ weakly$^*$ in } L^{\Psi}(\Go).\label{prpr2}
\end{align}
Then we set $q:=\phi^n-\phi_0- \tilde{q}$ in \eqref{wspot}, which is now an admissible choice to obtain
\begin{equation}\label{enernn}
\ii(\he(\nn\phi)-\je_0)\ccdot\nn(\phi^n-\phi_0 - \tilde{q})+\int_{\Go}b([\phi])\centerdot[\phi^n-\tilde{q}]= 0 \quad \textrm{ for all } \tilde{q}\in \PE.
\end{equation}
Since $\Psi^*$ and $\Phi^*$ satisfy $\Delta_2$ condition, we see that $\he(\nabla \phi)\in \EP{\phi^*}{\Omega}$ and $b([\phi])\in \Ep{\Psi}{\Go}$. Consequently, we can use \eqref{prpr1}--\eqref{prpr2} and let $n\to \infty$ in \eqref{enernn} to recover \eqref{ener}. Note that in both cases, we obtain \eqref{ener} even with the equality sign.

The second part of the proof, i.e. the part for $\je$, is done analogously and therefore is omitted here.

\subsection{Proof of Theorem~\ref{eq}}\label{SS2}
We start the proof with the claim i). If $\phi$ is a weak solution then it directly follows from \eqref{wspot} that $\je-\je_0 \in \X$ with $\je \cdot \n = b([\phi])$ on~$\Go$. Thus, it remains to check that \eqref{ws1} is satisfied. Hence, let $\ta \in \E$ be arbitrary. Then, using the definition of $\je$, we have
\begin{equation}\label{Va1}
\ii\fe(\je)\ccdot\ta+\int_{\Gamma}g(\je\cdot\n)\centerdot(\ta\cdot\n)-\ii\nn\phi_0\ccdot\ta =  \ii(\nn \phi -\nn \phi_0)\ccdot\ta+\int_{\Gamma}[\phi]\centerdot(\ta\cdot\n).
\end{equation}
Thus, if $\ta$ is in addition $\mathcal{C}^1$, then we can directly integrate by parts and we see that the right hand side vanishes, which finishes the first part of i). Second, assume that $(\Delta)$ holds. In the first case, i.e. if $\Phi$ and $\Psi$ satisfy $\Delta_2$ condition, then we have that $\phi-\phi_0\in \PE$ and the right hand side of \eqref{Va1} vanishes by using the definition of the space $\E$. In the second case, we use the fact that we can approximate $\phi$ by a proper sequence defined in \eqref{prpr1}--\eqref{prpr2} and we can write
$$
\ii(\nn \phi -\nn \phi_0)\ccdot\ta+\int_{\Gamma}[\phi]\centerdot(\ta\cdot\n) = \lim_{n\to \infty}\ii(\nn \phi^n -\nn \phi_0)\ccdot\ta+\int_{\Gamma}[\phi^n]\centerdot(\ta\cdot\n)=0,
$$
where the second equality follows from the fact that for each $n\in \mathbb{N}$, there holds $\phi^n-\phi_0 \in \PE$ and from the definition of the space $\E$. Hence, the integral in \eqref{Va1} vanishes, which is nothing else than \eqref{ws1}.

Next, we focus on the part ii). Hence, let $\je \in \je_0 + \X$ be a weak solution. Then we can set $\ta:=\ta_1$ in \eqref{ws1}, where $\ta_1 \in \mathcal{C}^{1}(\Omega; \R^{d\times N})\cap \E$ is arbitrary fulfilling $\ta_1\equiv 0$ in $\Omega_2$ to obtain
\begin{equation}\label{ws1i}
\int_{\Omega_1}(\fe(\je)-\nn \phi_0)\ccdot\ta_1=0.
\end{equation}
Consequently, the de~Rahm theorem implies that  there exist $\phi_i\in W^{1,1}(\Omega_i; \R^N)$, such that
$$
\fe(\je)=\nabla \phi_1 \quad \Leftrightarrow\quad \je =\he (\nabla \phi_1)\qquad \textrm{ in } \Omega_1.
$$
In addition, since $\partial \Omega_1\cap \Gamma_D\neq \emptyset$, we have from \eqref{ws1i} that $\phi_1$ must be chosen such that $\phi_1=\phi_0$ on $\partial \Omega_1\cap \Gamma_D$. Consequently, it is unique. Similarly, we can uniquely construct $\phi_2$ fulfilling $\phi_2=\phi_0$ on $\partial \Omega_2 \cap \Gamma_D$ and
$$
\fe(\je)=\nabla \phi_2 \quad \Leftrightarrow\quad \je =\he (\nabla \phi_2)\qquad \textrm{ in } \Omega_2.
$$
Thus, defining finally
$$
\phi:= \phi_1 \chi_{\Omega_1} + \phi_2\chi_{\Omega_2}
$$
and using the definition of a weak solution $\je$ and the fact that $\he = \fe^{-1}$, we deduce that (recall here that the notion of $\nabla$ does not reflect the jump over $\Go$)
$$
\ii \Phi(\nabla \phi) + \ii \Phi^* (\he (\nabla \phi))= \ii \Phi (\fe(\je)) +\ii \Phi^* (\je) <\infty.
$$
To identify also a jump $[\phi]$ on $\Go$, we first state the following result, which will be proven at the end of this section.
\begin{lemma}\label{Lvar} Let $\Omega$ satisfy (O1)--(O3) and $f\in L^1(\Go)$ be given. Assume  that for all $\ta \in \mathcal{C}^1(\overline{\Omega}_1;\R^d)$ fulfilling $\Div \ta=0$ in $\Omega_1$ and $\ta \cdot \n=0$ on $\Gamma_N \cap \partial \Omega_1$  there holds
\begin{equation}\label{58}
\int_{\Go}f\ta \cdot \n=0.
\end{equation}
Then $f\equiv 0$ almost everywhere on $\Go$.
\end{lemma}
The above lemma is used in the following way. We set  $\ta \in \E \cap \mathcal{C}^1(\overline{\Omega};\R^{d\times N})$ in \eqref{ws1} arbitrarily  and using the definition of $\phi$ and integration by parts, we find that
\begin{equation}\label{ws1k}
\begin{split}
0&=\ii(\fe(\je)-\nn \phi_0)\ccdot\ta+\int_{\Gamma}g(\je\cdot\n)\centerdot(\ta\cdot\n)\\
&=\ii \nn (\phi-\phi_0)\ccdot\ta+\int_{\Gamma}g(\je\cdot\n)\centerdot(\ta\cdot\n)\\
&=-\int_{\Go} [\phi-\phi_0]\centerdot(\ta\cdot\n)+\int_{\Gamma}g(\je\cdot\n)\centerdot(\ta\cdot\n)\\
&=\int_{\Go}(g(\je\cdot\n)- [\phi])\centerdot(\ta\cdot\n).
\end{split}
\end{equation}
Since $\ta$ was arbitrary, can use \eqref{58} to conclude
$$
[\phi]=g(\je\cdot\n) \quad \Leftrightarrow\quad b([\phi])=\je \cdot \n \qquad \textrm{ on } \Go.
$$
Consequently, we also have (by using of the notion of weak solution and the fact that $g=b^{-1}$)
$$
\int_{\Go} \Psi([\phi]) + \ii \Psi^* (b([\phi]))= \int_{\Go} \Psi (g(\je\cdot\n)) +\int_{\Go} \Psi^* (\je\cdot\n) <\infty.
$$
Finally, it directly follows from the definition of $\X$ and the identification of $\phi$ that it satisfies \eqref{wspot} and thanks to the above estimates $\phi$ is a weak solution. It just remains to prove Lemma~\ref{Lvar}.
\begin{proof}[Proof of Lemma~\ref{Lvar}]
We start the proof by considering arbitrary $\Gamma_i \subset \Gamma$, where $\Gamma_i$ can be described as a graph of Lipschitz function depending on the first $(d-1)$ spatial variables, i.e.  $x_1,\ldots, x_{d-1}$ (here we use the fact that $\Omega_1$ is Lipschitz) and fulfills for some cube $Q_{R_i}\subset \R^d$, $\Gamma_i \subset Q_{R_i} \subset Q_{2R_i} \subset \Omega$, where $Q_{R_i}:= \x_0 + (-R_i, R_i)^d$ with some $\x_0 \in \mathbb{R}^d$. Furthermore, we can require (this also follows from the Lipschitz regularity of $\Omega_1$ and from proper orthogonal transformation) that for some $\varepsilon>0$
\begin{equation}
\label{require}
\n \cdot (\underset{(d-1)\textrm{-times}}{\underbrace{0,\ldots,0}},1) \ge \varepsilon \qquad \textrm{on } \Gamma_i.
\end{equation}

Next, let $\psi \in \mathcal{C}^{\infty}_0 (\{ \x_0 + (-R_i,R_i)^{d-1}\})$ be arbitrary function depending only on $x_1,\ldots,x_{d-1}$ and $g\in \mathcal{C}^{\infty}_0 (Q_{2R_i})$ be arbitrary function fulfilling $g\equiv 1$ in $Q_{R_i}$. Then we set
$$
\ta_1:=(\underset{(d-1)\textrm{-times}}{\underbrace{0,\ldots,0}}, \psi(x_1,\ldots,x_{d-1}) g(x_1,\ldots, x_d)).
$$
Note that $\ta_1\in \mathcal{C}^{\infty}_0(\R^d; \R^d)$. Finally, since $\Omega_1$ is connected and $\Gamma_D$ has positive measure we can find a smooth open connected set $G\subset \R^d$ such that
$$
\begin{aligned}
\{x\in \Omega_1; \; \psi(x) \partial_{x_d}g(x) \neq 0\}&\subset  G,\\
\overline{G}\cap \partial \Omega_1 &\subset \Gamma_D,\\
G\setminus \overline{\Omega}_1 &\neq \emptyset.
\end{aligned}
$$
Finally, we find an arbitrary $h\in \mathcal{C}^{\infty}_0 (G\setminus \overline{\Omega}_1)$ such that
\begin{equation}
\int_{G\setminus \overline{\Omega}_1} h = -\int_{G\cap \Omega_1} \psi(x) \partial_{x_d}g(x). \label{compap}
\end{equation}
Next, we use the Bogovskii operator and we can find $\ta_2 \in \mathcal{C}^{\infty}_0 (G; \R^d)$ satisfying
$$
\Div \ta_2 =  \psi \partial_{x_d}g + h \qquad \textrm{in } G.
$$
Note that such function can be found due to the compatibility assumption \eqref{compap}. Furthermore, we simply extend $\ta_2$ by zero outside $G$. Having prepared $\ta_1$ and $\ta_2$, we set $\ta:=\ta_1-\ta_2$. Then it follows from the construction that in $\Omega_1$ (note that $h$ is not supported in $\Omega_1$)
$$
\Div \ta = \Div \ta_1 - \Div \ta_2 = \psi \partial_{x_d}g-\psi \partial_{x_d}g=0
$$
and that $\ta=\bs{0}$ on $\Gamma_N$. Consequently, $\ta$ can be used in \eqref{58} and we have
$$
0=\int_{\Gamma} f (\ta \cdot \n) = \int_{\Gamma_i} f \psi n_d.
$$
Since $\psi$ is arbitrary then $fn_d=0$ almost everywhere\footnote{Here in fact the function $\psi$ depends only on the first $(d-1)$ variables, but since the set $\Gamma_i$ is described as a graph of a Lipschitz mapping depending on $x_1,\ldots, x_{d-1}$, we can use the standard substitution and the fundamental theorem about integrable functions.} in $\Gamma_i$. Further, since $n_d >0$ everywhere on $\Gamma_i$ then
\begin{equation*}
f=0\quad\text{ on }\Gamma_i.
\end{equation*}
This statement holds true for arbitrary $\Gamma_i$ and therefore can be extended to the whole $\Gamma$. The proof is complete.
\end{proof}

\subsection{Proof of Theorem~\ref{ex1}}\label{SS3}
In this part, we assume that $(\Pi)$ holds, i.e. there exists $F_{\he}$ and $F_b$ such that for any $\ve\in\R^{d\times N}$ and $z\in\R^N$
$$
\frac{\partial F_{\he}(\ve)}{\partial \ve}=\he(\ve)\quad\text{and}\quad \frac{\partial F_{b}(z)}{\partial z}=b(z).
$$
Furthermore, since $\he$ and $b$ are coercive and monotone mappings (see \eqref{nula}--\eqref{ii}), it directly follows that $F_{\he}$ and $F_b$ are $N$-functions (non-negative, even, convex mappings). In addition, we evidently have the following identities for the G\^ateaux derivatives of $\he$ and $b$:
\begin{equation}\label{gat1}
\partial_{\uu}F_{\he}(\ve)\equiv\lim_{\la\to0_+}\f1{\la}(F_{\he}(\ve+\la\uu)-F_{\he}(\ve))=\he(\ve)\ccdot\uu,\quad\ve,\uu\in\R^{d\times N},
\end{equation}
and analogously
\begin{equation}\label{gat2}
\partial_yF_b(z)=b(z)y,\quad z,y\in\R^N.
\end{equation}
In addition, it follows from the definition of the convex conjugate function that we can replace \eqref{i}--\eqref{ii} by more sharp identities
\begin{align}
\he(\ve)\ccdot \ve &= F_{\he}(\ve)+ F^*_{\he}(\he(\ve)),\label{ihe}\\
b(z) \centerdot z&=F_b(z)+ F^*_b(b(z))\label{iihe}
\end{align}
and with the help of \eqref{ihe}--\eqref{iihe}, we can identify $\Phi$ and $\Psi$ from \eqref{i}--\eqref{ii} with $F_{\he}$ and $F_b$, i.e. we set in the rest of the proof $\Phi:=F_{\he}$ and $\Psi:=F_b$.
Finally, we define the following functional
\begin{equation}\label{funct}
I(p):=\ii F_{\he}(\nn (\phi_0+p)) - \je_0 \ccdot \nn(\phi_0+p)+\int_{\Go}F_b([p]) \qquad \textrm{ for all }p\in P
\end{equation}
and look for the minimizer, i.e. we want to find $p\in P$ such that for all $q\in P$ there holds
\begin{equation}
I(p)\le I(q) \qquad \Leftrightarrow \qquad I(p)=\min_{q\in P}I(q). \label{minch}
\end{equation}
To prove the existence of $p$ fulfilling \eqref{minch}, we define
$$
m:=\inf_{q\in P} I(q)
$$
and find $\set{p^n}_{n=1}^{\infty}$ as a minimizing sequence of $I$. It follows from the assumptions on $\phi_0$ and $\je_0$ that such a sequence can be found and it fulfils for all $n\in \mathbb{N}$
$$
I(p^n)\le 2I(0) <\infty.
$$
Hence, using the assumption on $\je_0$, the property \eqref{goodups} and  the Young inequality, and defining $\phi^n:=\phi_0+p^n$, we find that
\begin{equation}\label{qwer}
\begin{aligned}
&\ii F_{\he}(\nabla \phi^n) + \int_{\Go}F_b([\phi^n])\\
&\le 2\left(\ii F_{\he}(\nabla \phi^n) -\je_0 \ccdot \nn \phi^n + \int_{\Go}F_b([\phi^n])\right)  +2\ii F^*_{\he}(2\je_0) \\
&\le 4I(0)+2\ii F^*_{\he}(2\je_0) <\infty.
\end{aligned}
\end{equation}
Having such uniform bound, we can use the Banach-Alaoglu theorem, and find $\phi\in\phi_0+P$ and a subsequence, that we do not relabel, such that
\begin{equation}\label{spacess}
\begin{aligned}
\nn\phi^n&\wcs\nn\phi&&\inn\LP{\Phi}{\Omega},\\
[\phi^n]&\wcs[\phi]&&\inn L^{\Psi}(\Go)
\end{aligned}
\end{equation}
(there is no need to identify the weak limits since the operators of trace, $\nn$ and $[\cdot]$ are linear). Obviously, these two convergence results hold in the weak-$L^1$ topology as well (since $\Phi$ and $\Psi$ are superlinear). Thus, thanks to the convexity of $F_{\he}$ and $F_b$
and by the fact that
$$
F_{\he}(\nn (\phi_0+p)) - \je_0 \ccdot \nn(\phi_0+p) \ge -F^*_{\he}(\je_0) \in L^1(\Omega),
$$
we can use the weak lower semicontinuity of convex functionals to observe that
$$m=\lim_{n\to\infty}I(p^n)\geq I(p)\geq m,$$
hence $I(p)=I(\phi-\phi_0)=m$ is a minimum. Furthermore, it follows from \eqref{qwer} that
\begin{equation}\label{qwer2}
\begin{aligned}
&\ii F_{\he}(\nabla \phi^n) + \int_{\Go}F_b([\phi^n])<\infty.
\end{aligned}
\end{equation}

Now we will prove that $\phi$ is a variational weak solution. This will be done by deriving the Euler-Lagrange equation corresponding to $I$. Let $q\in \PE$ be arbitrary and denote $\phi_q:=\phi_0+q$. We set
\begin{align*}
D_{\he}(\la)&:=\f{F_{\he}(\nn\phi+\la(\nn \phi_q-\nn\phi))- F_{\he}(\nn\phi)}{\la}\\
D_{b}(\la)&:=\f{F_b([\phi]+\la([\phi_q]-[\phi]))-F_b([\phi])}{\la},
\end{align*}
where $\la\in(0,1)$ is arbitrary. Then, we use the minimizing property \eqref{minch} to get
$$
I(p)\le I((1-\lambda)p+\lambda \phi_q),
$$
which in terms of $D_{\he}$ and $D_b$ can be rewritten by using \eqref{funct} as
\begin{equation}\label{funct23}
-\ii  \je_0 \ccdot (\nn \phi - \nn\phi_0 -\nn q) \le \ii D_{\he}(\la) +\int_{\Go}D_b(\la).
\end{equation}
Next, \eqref{gat1} and \eqref{gat2} imply that (recall that $[\phi_0]=0$ on $\Go$)
\begin{align*}
D_{\he}(\la)&\to \he(\nn \phi) \ccdot \nn(q -\phi+\phi_0),\\
D_{b}(\la)&\to b([\phi])\centerdot [q -\phi]
\end{align*}
almost everywhere in $\Omega$ and $\Go$, respectively, as $\la \to 0_+$. Our goal now is to let $\lambda \to 0_+$ in \eqref{funct23}. Indeed, if we can justify the limit procedure in the term on the right hand side and if we use the above point-wise result, we directly obtain \eqref{ener}, i.e. $\phi$ is a variational weak solution. Then we can use the already proven Theorem~\ref{cons} to conclude that $\phi$ is also a weak solution. Hence, to finish the proof, we need to justify the limit procedure. Since, we  need to pass to the limit with the inequality sign, we use the Fatou lemma. Therefore we need to find $I_1 \in L^1(\Omega)$ and $I_2 \in L^1(\Go)$ such that for all $\lambda \in (0,1)$ we have
\begin{equation}\label{goalik}
D_{\he}(\la)\le I_1 \textrm{ in } \Omega \qquad \textrm{ and } \qquad D_b(\la)\le I_2 \textrm{ on } \Go
\end{equation}
and  that for all $\lambda\in (0,1)$ we have (possibly non-uniformly)
\begin{equation}\label{goalik2}
\ii D_{\he}(\lambda) >-\infty, \qquad \int_{\Go} D_b(\la) >-\infty.
\end{equation}
Thanks to nonnegativity of $F_{\he}$ and $F_b$, and due to \eqref{spacess} and \eqref{qwer2}, we get
$$
\begin{aligned}
\ii D_{\he}(\la)&\geq-\f1{\la}\ii F_{\he}(\nn\phi)>-\infty,\\
\int_{\Go} D_{b}(\la)&\geq-\f1{\la}\int_{\Go} F_{b}([\phi])>-\infty
\end{aligned}
$$
for all $\lambda \in (0,1)$, which is \eqref{goalik2}. To show also \eqref{goalik}, we use the convexity and the nonnegativity of $F_{\he}$, which  yields
\begin{align}\label{conev}
D_{\he}(\la)\leq\f{(1-\la) F_{\he}(\nn\phi)+\la F_{\he}(\nn \phi_q)- F_{\he}(\nn\phi)}{\la}&\le F_{\he}(\nn q + \nn \phi_0)
\end{align}
for all $\la\in(0,1)$.

To see that $I_1:=F_{\he}(\nn q+\nn \phi_0)\in L^1(\Omega)$, we use the assumption on $\phi_0$ and $q$. Since both $\nn q, \nn \phi_0\in \EP{\Phi}{\Omega}$, which is a linear space, we have that $\nn q+ \nn \phi_0 \in \EP{\Phi}{\Omega}$ as well. Consequently, we can use \eqref{goodups} to conclude that
$$
\ii I_1=\ii F_{\he}(\nn q+\nn \phi_0)<\infty,
$$
which leads to the first part of \eqref{goalik}. The second part is however proven similarly. Hence, we are allowed to use the Fatou lemma and to let $\la \to 0_+$ in \eqref{funct23} to obtain \eqref{ener}. This finishes the existence part of the proof. 
%
%
%

\subsection{Proof of Theorem~\ref{ex2}}\label{SS4}
We assume in this part that $(\Delta)$ holds. We proceed here as follows. First, we define the Galerkin approximation, then we derive uniform estimates and pass to the limit. Finally, depending on what kind of $\Delta_2$ condition is satisfied, we finish the proof.
\subsubsection{Galerkin approximation}
We know that $\E$ is a separable space, therefore we can find $\set{\we^i}_{i=1}^{\infty}\subset \E$, whose linear hull is dense in $\E$. Next, we construct an approximative sequence $\je^n$ in the following way. For $\af=(\alpha_1,\ldots,\alpha_n)\in\R^n$, we denote $\we_{\af}=\je_0+ \sum_{i=1}^n\alpha_i\we^i$.  Then we define the $i$-th component, $i\in\{1,\ldots,n\}$, of the mapping $\Fe$ by
\begin{equation}\label{6}
	\Fe_i(\af):=\ii\fe(\we_{\af})\ccdot\we^i+\int_{\Go}g(\we_{\af}\cdot\n)\centerdot(\we^i\cdot\n)-\ii\nn\phi_0\ccdot\we^i,\quad\af\in\R^n.
\end{equation}
Our goal is to find $\af^{\ast} \in \R^n$ such that $\Fe(\af^{\ast})=0$. Indeed, having such $\af^{\ast}$ is equivalent to have $\je^n:=	\je_0+ \sum_{i=1}^n\alpha^{\ast}_i\we^i$ such that
\begin{align}\label{gal}
	\qquad\ii\fe(\je^n)\ccdot\we^i+\int_{\Go}g(\je^n\cdot\n)\centerdot(\we^i\cdot\n)=\ii\nn\phi_0\ccdot\we^i \textrm{ for all } i\in\{1,\ldots,n\}.
\end{align}
Hence, we focus now on finding the zero point of $\Fe$ defined in \eqref{6}.
Since we assume that $\fe$ and $g$ are Carath\'{e}odory mappings and $\je_0 \in \EP{\Phi^*}{\Omega}$, we can use \eqref{goodups} to deduce that  the mapping $\Fe$ is continuous on $\R^{n}$. Moreover, using the growth properties of $\fe$ and $g$ (assumption (A2)$^*$), the Young inequality, the fact that $\je_0 \cdot \n = 0$ on $\Gamma$, $\je_0 \in \EP{\Phi^*}{\Omega}$ and also that $\nn\phi_0\in\EP{\Phi}{\Omega}$, we get
\begin{equation}	
\begin{aligned}\label{7}
	&\Fe(\af)\cdot\af :=\sum_{i=1}^n\Fe_i(\af)\alpha_i=\ii\fe(\we_{\af})\ccdot (\we_{\af}-\je_0)\\
&\quad +\io g(\we_{\af}\cdot\n)\centerdot(\we_{\af}\cdot\n)-\ii\nn\phi_0\ccdot (\we_{\af}-\je_0)\\
	&\geq \alpha_{\fe}\ii(\Phi^*(\we_{\af})+\Phi(\fe(\we_{\af}))+\alpha_g\io(\Psi^*(\we_{\af}\cdot\n)+\Psi(g(\we_{\af}\cdot\n))\\
	&\qquad-\f {\alpha_{\fe}}2\ii(\Phi^*(\we_{\af})+\Phi(\fe(\we_{\af}))\\
&\quad-2\ii\Phi\left(\f{2}{\alpha_{\af}}\nn\phi_0\right)+\Phi^*\left(\f{2}{\alpha_{\af}}\je_0\right)-C\\
	&\geq \frac{\alpha_{\fe}}{2}\ii(\Phi^*(\we_{\af})+\Phi(\fe(\we_{\af}))+\frac{\alpha_g}{2}\io(\Psi^*(\we_{\af}\cdot\n)+\Psi(g(\we_{\af}\cdot\n))-C.
	\end{aligned}
\end{equation}
	Since the mapping $\af\mapsto\we_{\af}$ is linear and since $\Phi^*$, $\Psi^*$ satisfy \eqref{grow}, there exists $R>0$ such that if $|\af|>R$, then $\Fe(\af)\cdot\af>1$. Hence, using a well known modification of the Brouwer fixed point theorem, there exists a point ${\af^{\ast}}\in\R^{n}$ with $\Fe({\af^{\ast}})=0$, which we wanted to show. Consequently, we also obtained the existence of $\je^n$ solving \eqref{gal}.

\subsubsection{Uniform estimates and limit $n\to \infty$}
It follows from \eqref{gal} (see the computation in \eqref{7}) that the identity
\begin{align}\label{enn}
	\ii\fe(\je^n)\ccdot(\je^n-\je_0)+\int_{\Go}g(\je^n\cdot\n)\centerdot(\je^n\cdot\n)&=\ii\nn\phi_0\ccdot(\je^n-\je_0)
\end{align}
is valid for all $n\in \mathbb{N}$. Consequently, it follows by the same procedure as in \eqref{7} that we have the following uniform bounds
\begin{equation}\label{ue2}
	\ii(\Phi^*(\je^n)+\Phi(\fe(\je^n)))+\int_{\Go}(\Psi^*(\je^n\cdot\n)+\Psi(g(\je^n\cdot\n)))\leq C.
	\end{equation}
Thus, using the Banach-Alaoglu theorem, we find weakly-$\ast$ converging subsequences (that we do not relabel), so that
	\begin{align}
	\je^n&\wcs\je &&\text{in}\quad \LP{\Phi^*}{\Omega},\label{14}\\
	\fe(\je^n)&\wcs \overline{\fe} &&\text{in}\quad\LP{\Phi}{\Omega},\label{12}\\
	\je^n\cdot\n&\wcs\je\cdot\n &&\text{in}\quad L^{\Psi^*}(\Go),\label{15}\\
	g(\je^n\cdot\n)&\wcs \overline{g} &&\text{in}\quad L^{\Psi}(\Go)\label{13}
	\end{align}
as $n\to\infty$. Furthermore, since $\je^n-\je_0 \in \E$, we have from the above convergence result that $\je-\je_0 \in \B$. Next, we pass to the limit also in \eqref{gal}.  Since $\we^i\in \EP{\Phi^*}{\Omega}$ and $\we^i\cdot\n\in\Ep{\Psi^*}{\Go}$ for all $i\in\N$, we can use \eqref{12} and  \eqref{13} to let $n\to \infty$ in \eqref{gal} for fix $i\in \N$ and obtain
	\begin{align}\label{gal2}
	\ii\overline{\fe}\ccdot\we^i+\int_{\Go}\overline{g}\centerdot(\we^i\cdot\n)&=\ii\nn\phi_0\ccdot\we^i \quad \textrm{for all }i\in\{1,\ldots,n\}
	\end{align}
and since the linear hull of $\{\we^i\}_{i\in\N}$ is dense in $\E$, we obtain
	\begin{align}\label{gal3}
	\ii\overline{\fe}\ccdot\ta+\int_{\Go}\overline{g}\centerdot(\ta\cdot\n)&=\ii\nn\phi_0\ccdot\ta \quad \textrm{ for all }\ta\in\E.
	\end{align}

\subsubsection{Identification of $\fe$ and $g$ and the energy (in)equality}
To finish the proof, it remains to show that
\begin{equation}
\label{fuck}
\overline{\fe}= \fe(\je) \textrm{ a.e. in } \Omega \qquad \textrm{ and } \quad \overline{g}=g(\je \cdot \n)\textrm{ a.e. on } \Gamma
\end{equation}
and also that we constructed the variational solution. We start the proof by claiming that
\begin{align}\label{ga23}
	\ii\overline{\fe}\ccdot (\je -\je_0)+\int_{\Go}\overline{g}\centerdot(\je \cdot\n)&=\ii\nn\phi_0\ccdot (\je -\je_0).
\end{align}
The importance of \eqref{ga23} is not only that it will allow us to show \eqref{fuck} but also that having \eqref{fuck}, \eqref{ga23} and \eqref{gal3}, we immediately get \eqref{wsf1} even with the equality sign.
	
Hence, we prove \eqref{ga23} provided that $(\Delta)$ holds. First, in case  that $\Phi^*$ and $\Psi^*$ satisfy the $\Delta_2$ condition then  $\E=\X$ and \eqref{gal3} can be tested by any $\ta\in\X$, in particular by $\je-\je_0$ and \eqref{ga23} follows. In the opposite case, i.e. if $\Phi$ and $\Psi$ satisfy the $\Delta_2$ condition, then we have from \eqref{12} and \eqref{13} that 	$\overline{\fe}\in \EP{\Phi}{\Omega}$ and $\overline{g} \in E^{\Psi}(\Go)$. Furthermore, it follows from \eqref{gal2} that for all $i\in N$
\begin{align}\label{gal67}
	\ii\overline{\fe}\ccdot(\je^i-\je_0)+\int_{\Go}\overline{g}\centerdot(\je^i\cdot\n)&=\ii\nn\phi_0\ccdot(\je^i-\je_0).
	\end{align}
But now, we can use the convergence results \eqref{14} and \eqref{15} (thanks to $\overline{\fe}\in \EP{\Phi}{\Omega}$ and $\overline{g} \in E^{\Psi}(\Go)$) and let $i\to \infty$ in \eqref{gal67} to obtain \eqref{ga23}.
Next, using the facts that $\nn\phi_0\in\EP{\Phi}{\Omega}$ and $\je_0 \in \EP{\Phi^*}{\Omega}$ and \eqref{14}--\eqref{13}, we can let $n\to \infty$ in \eqref{enn} to deduce
\begin{equation}
\begin{aligned}
	&\lim_{n\to\infty}\left(\ii\fe(\je^n)\ccdot\je^n+\int_{\Go}g(\je^n\cdot\n)\centerdot(\je^n\cdot\n)\right)\\
&=\lim_{n\to\infty}\left(\ii\nn\phi_0\ccdot(\je^n-\je_0) + \ii\fe(\je^n)\ccdot\je_0\right)\\
	&=\ii\nn\phi_0\ccdot(\je-\je_0) + \ii \overline{\fe}\ccdot\je_0\overset{\eqref{ga23}}=\ii\overline{\fe}\ccdot\je+\int_{\Go}\overline{g}\centerdot(\je \cdot\n). \label{limsup}
\end{aligned}
\end{equation}
Now we follow \cite{B}, see also \cite{Bul1}. Let $\ve \in L^{\infty}(\Omega; \R^{d\times N})$ and $z\in L^{\infty}(\Gamma; \R^N)$ be arbitrary. Using the monotonicity assumptions (A1)$^*$, we have
\begin{equation} \label{eq12}
\begin{aligned}
0&\le \lim_{n\to\infty}\ii(\fe(\je^n)-\fe(\ve))\ccdot (\je^n-\ve)+\int_{\Go}(g(\je^n\cdot\n)-g(z))\centerdot(\je^n\cdot\n-z)\\
&=\ii(\overline{\fe}-\fe(\ve))\ccdot (\je-\ve)+\int_{\Go}(\overline{g}-g(z))\centerdot(\je \cdot\n-z),
\end{aligned}
\end{equation}
where we used \eqref{14}--\eqref{13} and \eqref{limsup}. Finally, we closely follow \cite{Bul1,GMW12,Gw22} (see also \cite[Lemma 2.4.2.]{B} for similar procedure for more general monotone mappings). We define the sets
$$
\Omega_j:= \{x\in \Omega; \; |\je(x)|\le j\}, \quad \Gamma_j:=\{x\in \Gamma;\; |\je(x)\cdot \n(x)|\le j\}.
$$
Then for arbitrary $\varepsilon>0$, $\overline{\ve}\in L^{\infty}(\Omega; \R^{d\times N})$, $\overline{z}\in L^{\infty}(\Gamma; \R^N)$ and arbitrary $j\le k <\infty$, we set
$$
\ve:=\je \chi_{\Omega_k} -\varepsilon \overline{\ve}\chi_{\Omega_j},\qquad z:=\je \cdot \n \chi_{\Gamma_k} -\varepsilon \overline{z}\chi_{\Gamma_j}
$$
in \eqref{eq12}. Doing so, we obtain (using also the fact that $\fe(0)=g(0)=0$)
\begin{equation} \label{eq123}
\begin{aligned}
0&\le \ii(\overline{\fe}-\fe(\je \chi_{\Omega_k} -\varepsilon \overline{\ve}\chi_{\Omega_j}))\ccdot (\je (1-\chi_{\Omega_k}) +\varepsilon \overline{\ve}\chi_{\Omega_j})\\
&\quad +\int_{\Go}(\overline{g}-g(\je \cdot \n \chi_{\Gamma_k} -\varepsilon \overline{z}\chi_{\Gamma_j}))\centerdot((\je \cdot\n)(1-\chi_{\Gamma_k}) +\varepsilon \overline{z}\chi_{\Gamma_j})\\
&= \varepsilon \int_{\Omega_j}(\overline{\fe}-\fe(\je  -\varepsilon \overline{\ve}))\ccdot \overline{\ve} +\varepsilon \int_{\Gamma_j}(\overline{g}-g(\je \cdot \n  -\varepsilon \overline{z}))\centerdot \overline{z}\\
&+\int_{\Omega \setminus \Omega_k}\overline{\fe}\ccdot \je  +\int_{\Go\setminus \Gamma_k}\overline{g}\centerdot(\je \cdot\n).
\end{aligned}
\end{equation}
Thanks to \eqref{14} and \eqref{12} and since $|\Omega \setminus \Omega_k| \to 0$, $|\Gamma\setminus \Gamma_k|\to 0$ as $k\to \infty$, we can let $k\to \infty$ in \eqref{eq123} to deduce
\begin{equation*}
\begin{aligned}
0&\le \varepsilon \int_{\Omega_j}(\overline{\fe}-\fe(\je  -\varepsilon \overline{\ve}))\ccdot \overline{\ve} +\varepsilon \int_{\Gamma_j}(\overline{g}-g(\je \cdot \n  -\varepsilon \overline{z}))\centerdot \overline{z}.
\end{aligned}
\end{equation*}
Dividing by $\varepsilon$ and letting $\varepsilon \to 0_+$, using the definition of $\Omega_j$ and $\Gamma_j$ (leading to the fact that $\je$ and also $\je \cdot \n$ are bounded on the integration domain) and the fact that $\fe$ and $g$ are Carath\'{e}dory, we finally observe
\begin{equation*}
\begin{aligned}
0&\le \int_{\Omega_j}(\overline{\fe}-\fe(\je ))\ccdot \overline{\ve} +\int_{\Gamma_j}(\overline{g}-g(\je \cdot \n  )) \centerdot\overline{z}.
\end{aligned}
\end{equation*}
Setting
$$
\overline{\ve}:=- \frac{\overline{\fe}-\fe(\je )}{1+|\overline{\fe}-\fe(\je )|}\qquad \textrm{and} \qquad \overline{z}:= -\frac{\overline{g}-g(\je \cdot \n  )}{1+|\overline{g}-g(\je \cdot \n  )|}
$$
we deduce that \eqref{fuck} is valid almost everywhere in $\Omega_j$ (and $\Gamma_j$, respectively) for every $j\in \mathbb{N}$. Since $|\Omega \setminus \Omega_{j}|\to 0$ and $|\Gamma\setminus\Gamma_j|\to0$ as $j\to \infty$, it directly follows that \eqref{fuck} holds.

%
%

\subsubsection{Uniqueness}
We start the proof by claiming that \eqref{ws1} holds for all $\ta \in \B$. Indeed, if $\Psi^*$ and $\Phi^*$ satisfy $\Delta_2$ condition then $\E=\B$ and there is nothing to prove. On the other hand if $\Psi$ and $\Phi$ satisfy $\Delta_2$ condition, then we use the fact $\fe(\je)\in\LP{\Phi}{\Omega}=\EP{\Phi}{\Omega}$ and $g(\je\cdot\n)\in\Lp{\Psi}{\Go}=\Ep{\Psi}{\Go}$. Hence, for arbitrary $\ta \in \B$, we can find an approximating sequence $\{\ta_k\}_{k=1}^{\infty}\subset \E$ such that
$$
(\ta_k, \ta_k\cdot \n)\wcs(\ta,\ta\cdot\n)\inn\LP{\Phi^*}{\Omega}\times\Lp{\Psi^*}{\Go}.
$$
We replace  $\ta$ by $\ta_k$ in \eqref{ws1}  and let $k\to\infty$. Using the above weak start convergence result, we recover that \eqref{ws1} holds also for $\ta$.

Finally, assume that we have to solutions $\je_1,\je_2 \in \je_0 +\B$. Subtracting \eqref{ws1} for $\je_2$ from that one for $\je_1$ we have for all $\ta \in \B$
$$
\ii (\fe (\je_1)-\fe(\je_2))\ccdot \ta + \int_{\Go} (g(\je_1\cdot \n)-g(\je_2\cdot \n))\centerdot(\ta \cdot \n) =0.
$$
Setting finally $\ta:=\je_1-\je_2\in \B$ and using the strict monotonicity of $\fe$, we find that $\je_1=\je_2$ in $\Omega$, which finishes the uniqueness part.

\subsection{Proof of Theorem~\ref{ex1}- case $(\Delta)$ holds}\label{SS5}
	This proof is analogous to the preceding proof of Theorem~\ref{ex2}~(i). Again, we approximate the problem using separability of $\PE$ and the Galerkin method. Eventually, we construct an approximation $\phi^n$ satisfying
	$$\ii\he(\nn\phi^n)\ccdot\nn q+\int_{\Go}b([\phi^n])\centerdot[q]=\ii\je_0\ccdot\nn q$$
	for all $q$ from some $n$-dimensional subspace of $\PE$. Then, using the analogous a~priori estimate to \eqref{ue2} and very similar limiting procedure, we let $n\to\infty$ and obtain \eqref{wspot}.

In addition, it is evident, that we obtain a weak solution $\phi \in \phi_0 + \PB$, which is the last claim of Theorem~\ref{ex1}. Furthermore, assume that $q\in \PB$ is arbitrary. Therefore it can be approximated by a weakly star convergent sequence $\{q^n\}_{n=1}^{\infty}\subset\PE$. Since $\he(\nabla \phi)\in \EP{\Phi^*}{\Omega}$ and $b([\phi])\in \Ep{\Psi^*}{\Go}$, we can now use \eqref{wspot}, where we replace $q$ by $q^n$ and using the weak star convergence, we can conclude that \eqref{wspot} holds even for all $q\in \PB$. Finally, assume that we have to solutions $\phi_1, \phi_2 \in \phi_0 + \PB$. Then using \eqref{wspot} and the above argument, we can deduce that
$$
\ii (\he(\nn \phi_1)-\he(\nn\phi_2))\ccdot \nn q + \int_{\Gamma} (b([\phi_1])-b([\phi_2]))\centerdot [q] =0.
$$
Hence, setting $q:=\phi_1-\phi_2 \in \PB$ in the above identity, we observe with the help of the strict monotonicity of $\he$  that
$$
\nn \phi_1 = \nn \phi_2 \textrm{ in }\Omega.
$$
Hence, since $\phi_1=\phi_2$ on the sets $\Gamma_D^1\subset \partial \Omega_1$, $\Gamma_D^2\subset \partial \Omega_2$ of positive measure, we see that $\phi_1=\phi_2$ in $\Omega_1$ and also in $\Omega_2$ and the solution is unique in the class $\phi_0 + \PB$.
%
%
%

\subsection{Proof of Theorem~\ref{ex2}~(ii)}\label{SS6}
This proof is analogous to the proof of Theorem~\ref{ex1}~(ii). Indeed, it is easy to see that if we define
$$
I(\ta):=\ii (F_{\fe}(\je_0+\ta) - \nn \phi_0 \ccdot \ta) + \int_{\Go} F_g(\je\cdot \n),\quad\ta\in\X,
$$
we can proceed as before  to get a minimum $\ta\in\X$ and the corresponding $\je:=\je_0 +\ta$  satisfying \eqref{wsf1}. This minimum is a weak solution by Theorem~\ref{cons}.

\subsection{Proof of Corollary~\ref{coro}.}\label{SS7}
We only need to prove that the nonlinearities defined in \eqref{const} satisfy all the assumptions of Theorem~\ref{ex1}. Namely, we show that ($\Pi$) holds and that $(\Delta)$ is valid. We define,
	\begin{equation}\label{exp}
	\Phi(\ve)=F_{\he}(\ve):= \cosh (|\ve|)-1, \qquad \Psi(z)=F_b(z):=\exp(|z|) -|z|-1.
	\end{equation}
It is clear that both functions are $N$-functions. Moreover, by a direct computation, we have that
$$
\frac{\partial F_{\he}(\ve)}{\partial \ve} = \f{\sinh|\ve|}{|\ve|}\ve, \qquad \frac{\partial F_{b}(z)}{\partial z} = \f{\exp (|z|)-1}{|z|}z
$$
and thus ($\Pi$) holds. Moreover, $F_{\he}$ and $F_b$ are strictly convex. Hence, we use Theorem~\ref{ex1} to get the existence of a weak solution.

To prove also further properties, we show that $\Psi^*$ and $\Phi^*$ satisfy $\Delta_2$ condition and consequently $(\Delta)$ holds as well and having such property, we can even prove uniqueness of a weak solution. First, one can easily observe that there exists $K>1$ such that
\begin{align*}
2K\Phi(\ve)&\le \Phi(2\ve) \textrm{ for all }\ve \in \R^{d\times N}, \, |\ve|\ge 1,\\
2K\Psi(z)&\le \Psi(2z) \textrm{ for all } z\in \R^N, \, |z|\ge 1.
\end{align*}
Then, by \cite[Theorem~4.2.]{Kras}, this implies that $\Phi^*$ and $\Psi^*$ satisfy the $\Delta_2$ condition. The proof is complete.



\def\cprime{$'$}
\providecommand{\bysame}{\leavevmode\hbox to3em{\hrulefill}\thinspace}
\providecommand{\MR}{\relax\ifhmode\unskip\space\fi MR }
\providecommand{\MRhref}[2]{%
  \href{http://www.ams.org/mathscinet-getitem?mr=#1}{#2}
}
\providecommand{\href}[2]{#2}

\end{document}